\documentclass[oneside,reqno,12pt]{amsart}
\usepackage[greek,english]{babel}
\usepackage{amsthm}
\usepackage{amsbsy}
\usepackage{amsfonts}
\usepackage{graphicx}
\usepackage{hyperref}
\usepackage{dsfont}
\hypersetup{colorlinks=true, citecolor=blue, linkcolor=red}

\newtheorem{thm}{Theorem}
\newtheorem{cor}{Corollary}
\newtheorem{lem}{Lemma}
\newtheorem{pro}{Proposition}
\newtheorem{rem}{Remark}
\newtheorem{definition}{Definition}

\title[Nondegeneracy for strongly competitive Gross-Pitaevskii systems]{Quantitative linear nondegeneracy of  approximate solutions to strongly competitive Gross-Pitaevskii systems in general domains in $N\geq 1$ dimensions}
\author{Christos Sourdis}
\address{General Lyceum of Moires, Heraklion, Crete, Greece.}
\email{sourdis@uoc.gr}

\begin{document}

\maketitle
\begin{abstract}
We consider strongly coupled competitive elliptic systems of Gross-Pitaevskii type that arise in the study of two-component Bose-Einstein condensates, in general smooth bounded domains of $\mathbb{R}^N$, $N\geq 1$. As the coupling parameter tends to infinity, solutions that remain uniformly bounded are known to converge to a segregated limiting profile, with the difference of its components satisfying a limit scalar PDE. Under natural non-degeneracy assumptions on a solution of the limit problem,  we show that the linearization of the Gross-Pitaevskii system around a 'sufficiently good' approximate solution does not have a kernel and obtain an estimate for its inverse with respect to carefully chosen weighted norms. Our motivation is the study of the persistence of solutions of the limit scalar problem for large values of the coupling parameter which is known only in two dimensions or if the domain has radial symmetry. \end{abstract}

\section{Introduction}
\subsection{The problem}
We are concerned with positive solutions to the
elliptic problem:
\begin{equation}\label{eqEq}
\left\{\begin{array}{l}
  - \Delta u_1 = f(u_1, x) - \beta u_1u_2
^2
\ \textrm{in}\ \Omega, \\
  - \Delta u_2 = f(u_2, x) - \beta u_2u_1^2
\ \textrm{in}\ \Omega, \\
  u_1 = u_2 = 0\ \textrm{on}\ \partial \Omega,
\end{array}\right.
\end{equation}
for sufficiently large values of the parameter $\beta>0$. Here $\Omega$ is a bounded, open domain in $\mathbb{R}^N$, $N\geq 2$, with smooth boundary,
$f :
 \mathbb{R} \times \Omega  \to \mathbb{R}$ is sufficiently smooth and odd in the first variable, i.e.,
\begin{equation}\label{eqfodd} f(u, x) \equiv -f(-u, x).\end{equation}
We also assume that $f$ \emph{separates phases} in the following sense:
\begin{definition}\label{defSep} We say that the function $f \in C^1(\mathbb{R} \times \Omega)$ separates phases in $\Omega$ if the problem \begin{equation}\label{eqw}
 - \Delta w = f(w, x)\ \textrm{in}\  \Omega,\  w = 0\ \textrm{on}\ \partial \Omega, \end{equation}
has a solution $w$ such that $\Gamma = \left\{x\ | \ w(x) = 0\right\} \subset \Omega$ is    an $(N - 1)$-dimensional compact, smooth and embedded manifold    dividing $\Omega$ into
two disjoint, open components $\Omega_i$, $i = 1, 2$, with $\partial \Omega_i \cap \Omega = \Gamma$, $\Omega = \Gamma \cup \Omega_1\cup \Omega_2$ and
\begin{equation}\label{eqwHopf} \partial_\nu w(x) := \mu(x) > 0\ \textrm{on}\ \Gamma,\end{equation}
where $\nu$ denotes the choice of the unit normal to $\Gamma$ that is exterior to a fixed component of $\Omega$.
\end{definition}
To motivate the above definition, let us suppose that $\nu$ is the exterior to $\Omega_2$ so that $w > 0$ in $\Omega_1$ and
$w < 0$ in $\Omega_2$. Now, recalling that $f$ is odd in the first variable, we have
\[-\Delta w = f(w, x)\ \textrm{in}\ \Omega_1\ \textrm{and}\ - \Delta(-w) = f(-w, x)\ \textrm{in}\ \Omega_2.\]
Hence, the vector function
\begin{equation}\label{eqw0} \textbf{w}^0 = (w_1, w_2),\ \textrm{with}\ w_1(x) = w(x)\mathds{1}_{\Omega_1}(x)\ \textrm{and}\ w_2(x) = -w(x)\mathds{1}_{\Omega_2}(x),\end{equation}
would be a smooth solution to (\ref{eqEq}) if not for the jump of its derivative across $\Gamma$.
In the language of singular perturbation theory, keep in mind that $\beta\gg 1$,  $\textbf{w}^0$ is referred to as the \emph{outer solution} to (\ref{eqEq}).
Nevertheless, based on the presence of a suitable \emph{inner solution} which we shall describe shortly and for $f$ and $\Gamma$ sufficiently smooth, the persistence of $\textbf{w}^0$ in the uniform topology, for large $\beta$, has recently been established if $N=2$ in \cite{kowalczyk} (see also \cite{casteras} for the radial case with $N\geq 1$) under the following  non-degeneracy assumption:

\begin{definition}\label{defNonDeg}
We say that the phase separating solution $\textbf{w}^0$ to the problem (\ref{eqw}) is non-degenerate
if:
\begin{description}
  \item[(a)] Each of the following linear problems
  \begin{equation}\label{ewlinearSide}
                                        -\Delta \psi = f_u(w_i,x)\mathds{1}_{\Omega_i}\psi\  \textrm{in}\ \Omega_i; \
                                        \psi = 0\ \textrm{on}\ \partial \Omega_i,\ i = 1, 2,
   \end{equation} has only the trivial solution.
  \item[(b)]The linear problem \begin{equation}\label{ewlinearFul}
                                        -\Delta \psi = \left[f_u(w_1,x)\mathds{1}_{\Omega_1}+f_u(w_2,x)\mathds{1}_{\Omega_2}\right]\psi\  \textrm{in}\ \Omega; \
                                        \psi = 0\ \textrm{on}\ \partial \Omega,
   \end{equation} has only the trivial solution.
\end{description}
\end{definition}
It was conjectured in \cite{kowalczyk} that this persistence result can be generalized for any $N\geq 3$, as  in the radially symmetric case that was treated in \cite{casteras}.
Motivated in part by this, we are going to extend here the linear a-priori estimate \cite[Prop. 4.1]{casteras} to the above nonradial setting.
The aforementioned result concerns the linearization of (\ref{eqEq}) on a sufficiently good approximate solution (a $\beta$-dependent correction of $\textbf{w}^0$ near $\Gamma$, roughly speaking)
and its invertibility properties between carefully chosen weighted spaces.

\subsection{Ansatz for the approximate solution}\label{subsecAnsatz}Here, mainly motivated by \cite{casteras,kowalczyk}, we shall introduce our ansatz for an approximate solution of (\ref{eqEq}) in the  $\beta\gg 1$ regime. We stress that in the current work we will just be concerned with the study of the linearized problem on this ansatz; we postulate that   such an approximate solution can be derived through a matched asymptotic analysis, making use of both nondegeneracy conditions in Definition \ref{defNonDeg} (see also \cite{sourdisGelfand} for a problem that shares some similar features).

Throughout the paper, our approximate solution $\textbf{U}=(U_1,U_2)$ will be at least twice differentiable   in $\Omega$ with  nonnegative components.
Furthermore, for
notational reasons, it is convenient to introduce a small parameter
\begin{equation}\label{eqepsilonDef}
\epsilon=\frac{1}{\sqrt[4]{\beta}}.
\end{equation}

\subsubsection{Local coordinates near $\Gamma$}\label{subsectfermi}
In order to define our approximate solution near $\Gamma$, we need to use the natural   coordinates associated to it. Let $r=r(x)$ be the signed distance satisfying $r<0$
inside $\Gamma$ and $r>0$ outside $\Gamma$. Let $s = s(x)$ be the projection of $x$ on $\Gamma$ along the normal of $\Gamma$. Then,
there exists a $d_0 > 0$ such that $\Gamma(2d_0)= \left\{x \in \mathbb{R}^N\ :\  |r(x)| < 2d_0\right\} \subset \Omega$ and $\tau :
\Gamma(2d_0)\to (-2d_0,2d_0) \times \Gamma$ defined by $\tau(x) = (r(x), s(x))$ is a diffeomorphism.

\subsubsection{The inner limit ODE system}\label{subsectBerest}
In \cite{berestArma}, it was   shown that there exists a unique solution to the system
\begin{equation}\label{eqBerest}
 \left\{\begin{array}{c}
          -V_1''+V_1V_2^2=0\ \textrm{in}\ \mathbb{R}, \\
          -V_2''+V_2V_1^2=0\ \textrm{in}\ \mathbb{R},
        \end{array}
  \right.
\end{equation}such that $V_1$, $V_2>0$ in
$\mathbb{R}$ and $V_1(0) = V_2(0) = 1$, $V_1(t) = V_2(-t)$, $V '_1(t) > 0$,
\begin{equation}\label{eqBerestAsympt}
\begin{array}{l}
  V_1(t)=At+B+\mathcal{O}(e^{-ct^2})\ \textrm{as}\ t\to +\infty, \\
  V_1(t)=\mathcal{O}(e^{-ct^2})\ \textrm{as}\ t\to -\infty,
\end{array}
\end{equation}
for some $A>0$, $B\in \mathbb{R}$ and $c>0$. In fact, it is not hard to show that $B>0$ holds (see \cite[Sec. 2]{aftalionSour}).
Moreover, the above estimates in (\ref{eqBerestAsympt}) can be differentiated in the obvious way. Note also that (\ref{eqBerest}) is invariant under translation $\tau \to \textbf{V} (\cdot -\tau)$ and scaling $\lambda \to \lambda \textbf{V} (\lambda \ \cdot)$.

\subsubsection{The approximate solution in the inner region $  \Gamma(2|\ln \epsilon|\epsilon)$}\label{subsubΙΝAns}
In $ \Gamma(2|\ln \epsilon|\epsilon)$, we assume that our approximate solution $\textbf{U}=(U_1,U_2)$ satisfies
\begin{equation}\label{eqinner}
U_i(x)=\epsilon b(s) V_i\left(\epsilon^{-1}b(s)\left(r-\epsilon \zeta(s) \right) \right)+\epsilon^2 H(0,s)  W_i\left(\epsilon^{-1}b(s)\left(r-\epsilon \zeta(s) \right) \right)+Q_i(x),
\end{equation}
$i=1,2$, where $\textbf{W}=(W_1,W_2)$ is a solution to $\mathcal{M}(\textbf{W})=\textbf{V}'$ in $\mathbb{R}$, $\mathcal{M}$ as in (\ref{eqM}) below,   such that
\begin{equation}\label{eqWdiaz2}
\textbf{W}(z)=-\frac{1}{2}z^2\textbf{V}'(z)+\mathcal{O}(e^{-cz^2}), \ W_1(z) = -W_2(-z),
\end{equation}
for some constant $c>0$ (see \cite[Lem. 2.3]{kowalczyk}), and $H$  as in (\ref{eqMeanCurv}) below;

\begin{equation}\label{eqb}
  b(s)=\sqrt{\frac{\mu(s)}{A}}+\epsilon \tilde{b}(s),\ |\tilde{b}|+|\nabla_\Gamma \tilde{b}|+\left|\partial_\sigma\left(\nabla_\Gamma \tilde{b}\right)\right|\leq C,
\end{equation}
\begin{equation}\label{eqzeta}
   |\zeta|+|\nabla_\Gamma \zeta|+\left|\partial_\sigma\left(\nabla_\Gamma \zeta\right)\right|\leq C,
\end{equation}

\begin{equation}\label{eqinnerremalg}
  \left|Q_1(r,s)\right|\leq C(r^3+\epsilon^3),
\end{equation}

\begin{equation}\label{eqinnerremalgrr}
  \left|\partial_rQ_1(r,s)\right|\leq C(r^2+\epsilon^2),\ \ \left|\partial^2_{rr}Q_1(r,s)\right|\leq C(r+\epsilon),\ \ \left|\partial_r\nabla_\Gamma Q_1\right|\leq C(r^2+\epsilon^2),
\end{equation}

\begin{equation}\label{eqinnerremalgss}
  \left|\nabla_\Gamma Q_1\right| +\left|\partial_\sigma\left(\nabla_\Gamma Q_1\right)\right|\leq C(r^3+\epsilon^3),
\end{equation}

 \begin{equation}\label{eqQ2}
 |Q_2|+\epsilon |\partial_rQ_2|+\epsilon^2|\partial_{rr}^2Q_2|+|\nabla_\Gamma Q _2|+\left|\partial_\sigma\left(\nabla_\Gamma Q_2\right)\right|+\epsilon \left|\partial_r\nabla_\Gamma Q_2\right|\leq C_D\epsilon^2 e^{-Dr/\epsilon},
 \end{equation}
for any $D>0$, if $r\in [0,2|\ln \epsilon|\epsilon)$, $s\in \Gamma$, provided that $\epsilon>0$ is sufficiently small; here and throughout the paper  $C>0$ denotes a generic constant that is independent of small $\epsilon>0$ (the positive constants $D$ and $C_D$ in (\ref{eqQ2}) are also independent of $\epsilon$), $\nabla_\Gamma=\nabla-\nabla r(\nabla r \cdot \nabla)$ is the tangential gradient along $\Gamma$, and $\sigma$ is  any tangential direction to $\Gamma$. For $r<0$, our assumptions are analogous.

\subsubsection{The approximate solution in the outer region $\Omega\setminus \Gamma(|\ln \epsilon|\epsilon)$}\label{subsubOutAns} Let $w_i,\ \Omega_i$, $i=1,2$, be as in Definition \ref{defSep} and (\ref{eqw0}). Without loss of generality, we assume that $\Omega_1$ is the exterior subdomain (defined in the obvious way).
In $\left\{\Omega\setminus \Gamma(|\ln \epsilon|\epsilon)\right\}\cap \Omega_1$, we assume that our approximate solution has the form $\textbf{U}=(U_1,U_2)$ with
\begin{equation}\label{eqoutdef1}
U_1(x)=w_1(x)+\mathcal{O}_{C^2}(\epsilon),\ U_2(x)= \mathcal{O}_{C^2}(\epsilon^\infty),
\end{equation}
where the above remainders and their   derivatives up to the second order are uniformly bounded by $\epsilon$ and any algebraic power of $\epsilon$, respectively. Naturally, we assume that the Dirichlet boundary conditions are satisfied on $\partial \Omega\cap \overline{\Omega_1}$.
In $\left\{\Omega\setminus \Gamma(|\ln \epsilon|\epsilon)\right\}\cap \Omega_2$, we assume   the obvious corresponding properties.

\subsection{The linearized operator on the approximate solution} The linear operator associated to the linearization of (\ref{eqEq}) on the approximate solution $\mathbf{U}=(U_1,U_2)$ is
\begin{equation}\label{eqL}
  \mathcal{L}\left(\begin{array}{c}
               \varphi_1 \\

               \varphi_2
             \end{array}\right)=\left(\begin{array}{c}
              -\Delta \varphi_1 +\epsilon^{-4}U_2^2\varphi_1+2\epsilon^{-4}U_1U_2\varphi_2-f_u(U_1,x)\varphi_1\\

               -\Delta \varphi_2+\epsilon^{-4}U_1^2\varphi_2+2\epsilon^{-4}U_1U_2\varphi_1-f_u(U_2,x)\varphi_2
             \end{array}\right)
\end{equation}
with $\varphi_1,\varphi_2$ smooth and vanishing on $\partial \Omega$.

Our objective is to derive a useful a-priori estimate for the inhomogeneous problem
\begin{equation}\label{eqinhomog}
 \mathcal{L}( {\varphi})= {g}\  \textrm{in}\ \Omega; \  {\varphi}= {0}\ \textrm{on}\ \partial{\Omega},
\end{equation}
with respect to some carefully chosen weighted norms.

Our norm for $\varphi=(\varphi_1,\varphi_2)$ will be
\begin{equation}\label{eqNorm0}\begin{array}{rcl}
                                 \|\varphi\|_0 & = & \|\varphi\|_{L^\infty(\Omega)}+\|\varphi\|_{C^2\left(\Omega\setminus\Gamma(d)\right)}+\|\nabla_\Gamma\varphi\|_{L^\infty\left(\Gamma(d)\right)}+\epsilon\|\varphi_r\|_{L^\infty\left(\Gamma(d)\right)} \\
                                  &  & \\ & & +\sum_\sigma\|\partial_\sigma\left(\nabla_\Gamma \varphi\right)\|_{L^\infty\left(\Gamma(d)\right)}+\epsilon\|\nabla_\Gamma(\varphi_r)\|_{L^\infty\left(\Gamma(d)\right)}+\epsilon^2\|\varphi_{rr}\|_{L^\infty\left(\Gamma(d)\right)},
                               \end{array}
\end{equation}
where $d\in (0,d_0)$ is independent of $\epsilon$ (recall the appropriate definitions from Subsections \ref{subsectfermi} and \ref{subsubΙΝAns}). On the other side, our norm for $g=(g_1,g_2)$ will be
\begin{equation}\label{eqnorm1}
\begin{array}{rcl}
  \|g\|_1 & = & \|g\|_{C^1\left(\Omega \setminus \Gamma(d/2) \right)}\\
 &&\\ && +\left\|w(r)\left(|g_1|+|\partial_rg_1|+|\partial^2_{rr}g_1|+|\partial_r\nabla_\Gamma g_1|+|\nabla_\Gamma g_1|+\sum_{\sigma}|\partial_\sigma\left(\nabla_\Gamma g_1\right)|\right)\right\|_{L^\infty\left( \Gamma(2d) \right)} \\
   &  &  \\
   &  & +\left\|w(-r)\left(|g_2|+|\partial_rg_2|+|\partial^2_{rr}g_2|+|\partial_r\nabla_\Gamma g_2|+|\nabla_\Gamma g_2|+\sum_{\sigma}|\partial_\sigma\left(\nabla_\Gamma g_2\right)|\right)\right\|_{L^\infty\left( \Gamma(2d) \right)}\\
   & & \\
   & &+\epsilon^{-(1+\alpha)}\sum_{i=1}^{2}\|g_i\|_{L^\infty\left(\Omega_i \setminus \Gamma(2d) \right)}+e^{1/\epsilon}\sum_{i\neq j}\|g_i\|_{L^\infty\left(\Omega_j \setminus \Gamma(2d) \right)},
\end{array}
\end{equation}
where
\begin{equation}\label{eqweight}
w(r)=\left\{\begin{array}{ll}
              1+\left(\frac{r}{\epsilon}\right)^{1+\alpha}, & r\in [0,2d], \\
               &  \\
              e^{\frac{2|r|}{\epsilon}}, & r\in [-2d,0],
            \end{array}
 \right.
\end{equation}
for some $\alpha \in (0,1)$ that is independent of $\epsilon$.

\subsection{Main result} Our main result is the following.
\begin{thm}\label{thmMain}
Assume that $f, \Omega, \textbf{U}$ are as above, not necessarily satisfying the nondegeneracy condition (\ref{ewlinearSide}).
There exist positive constants $\epsilon_0$, $C$ such that if $\varphi,\ g \in C^2(\Omega) \cap C(\bar{\Omega})$ and $\epsilon \in (0,\epsilon_0)$ satisfy (\ref{eqinhomog}), then
\begin{equation}\label{eqaprioriThm}
\|\varphi\|_0\leq C\epsilon \|g\|_1.
\end{equation}
\end{thm}

\section{Proof of the main result}
\subsection{The Laplacian in local coordinates near $\Gamma$}\label{subsecFermiLaplac}
For $r\in (-2d_0,2d_0)$, let $\Gamma_r$ denote the part of the boundary of $\Gamma(r)$ which contains $(r,s)$ with $s\in \Gamma$. In particular, $\Gamma_0$ is $\Gamma$ itself.
The mean curvature of $\Gamma_r$ at
$(r, s)$ is given by
\begin{equation}\label{eqMeanCurv} H(r, s) =\sum_{i=1}^{N-1}
\frac{\kappa_i(s)}{1 - r\kappa_i(s)},
\end{equation}
where $\kappa_1,\cdots, \kappa_{N-1}$ stand for the principal curvatures of $\Gamma$ at $(0, s)$. Let us note that in the local coordinates $(r,s)$ near $\Gamma$ the Euclidean Laplacian is expressed as
\begin{equation}\label{eqLaplaceFermi1}\Delta_{\mathbb{R}^N}=\partial_{rr}-H(r,s)\partial_r+\Delta_{\Gamma_r},\end{equation}where $\Delta_{\Gamma_r}$ denotes   the Laplace-Beltrami operator of $\Gamma_r$ under the induced metric (see for instance \cite{mazzeo}). We also note that
\begin{equation}\label{eqLaplaceFermi2}
\Delta_{\Gamma_r}-\Delta_{\Gamma}=rB(r,s),
\end{equation}
where
\begin{equation}\label{eqLaplaceFermi3}
  B(r,s)=\sum_{i=1}^{N-1}b^i(r,s)\frac{\partial }{\partial s_i}+\sum_{i,j=1}^{N-1}b^{ij}(r,s)\frac{\partial^2 }{\partial s_i\partial s_j}
\end{equation}
with smooth coefficients $b^i, b^{ij}$. Moreover, the above relations can be differentiated in the obvious way.

\subsection{The Laplacian in the stretched, scaled and shifted coordinates $(z,s)$}\label{subsecLaplzt}
For $(r,s)\in \Gamma(2d_0)$, we define
\begin{equation}\label{eqzetaCoord}
z=\epsilon^{-1}b(s)\left(r-\epsilon \zeta(s)\right)
\end{equation}
(keep in mind (\ref{eqinner}), (\ref{eqb}) and (\ref{eqzeta})).
Given a smooth function $\varphi$ in $\Gamma(2d_0)$, we can define a new function $\Phi$ as follows
\begin{equation}\label{eqphiStretch}
\varphi(r,s)=\Phi(z,s).
\end{equation}

We have
\[
\varphi_r=\epsilon^{-1}b\Phi_z,\ \ \varphi_{rr}=\epsilon^{-2}b^2\Phi_{zz},
\]
\[
\nabla_\Gamma\varphi=(b^{-1}z\nabla_\Gamma b-b\nabla_\Gamma\zeta)\Phi_z+\nabla_\Gamma\Phi,
\]
and
\[\begin{array}{rcl}
    \Delta_\Gamma\varphi & = & (b^{-1}z\nabla_\Gamma b-b\nabla_\Gamma\zeta)^2\Phi_{zz}+\Delta_\Gamma\Phi+2(b^{-1}z\nabla_\Gamma b-b\nabla_\Gamma\zeta)\nabla_\Gamma\Phi_{z} \\
     & &\\
     &  & +(b^{-1}z\Delta_\Gamma b-2\nabla_\Gamma b\nabla_\Gamma\zeta-b\Delta_\Gamma\zeta)\Phi_{z}.
  \end{array}
\]

Hence, in light of (\ref{eqb}), (\ref{eqzeta}), (\ref{eqLaplaceFermi1}), (\ref{eqLaplaceFermi2}) and (\ref{eqLaplaceFermi3}), we deduce that
\begin{equation}\label{eqLaplaceFermi41}
                                          \begin{array}{rcl}
                                              b^{-2}\epsilon^2\Delta \varphi & = & \Phi_{zz}+b^{-2}\epsilon^2\Delta_\Gamma \Phi+\epsilon^2 \mathcal{O}(1+z^2)\Phi_{zz}+\epsilon^2\mathcal{O}\left(1+|z| \right)\nabla_\Gamma\Phi_z \\
                                              &   &   \\
                                              &   &  +\epsilon^3\mathcal{O}\left(1+|z|\right)\nabla^2_\Gamma \Phi  +\epsilon \mathcal{O}\left(1\right)  \Phi_z+\epsilon^2 \mathcal{O}\left(1\right)\nabla_\Gamma \Phi,   \\
                                          \end{array}
                                        \end{equation}
where by $\mathcal{O}(\cdot)$ we denote a generic function which satisfies
\begin{equation}\label{eqOss}
\left|\mathcal{O}(Z,s)\right|+\left|\nabla_\Gamma\mathcal{O}(Z,s)\right|+\left|\nabla^2_\Gamma\mathcal{O}(Z,s)\right|+|Z|\left|\partial_z\mathcal{O}(Z,s)\right|\leq C |Z|.
\end{equation}
\subsection{A local blow-up analysis near $\Gamma$}\label{subseqblowupequil}The following lemma will be applied to solutions of (\ref{eqinhomog}) and their tangential derivatives throughout the proof of Theorem \ref{thmMain}.
Related results in different contexts can be found in \cite[Lem. 6.2]{delAnnals}  and \cite[pgs. 3456-3457]{delPV24}.

\begin{lem}\label{lemBU}
Assume that $\varphi_n \in \left[C^2\left(\Gamma(d) \right)\right]^2$ and $g_n \in \left[C\left(\Gamma(d) \right)\right]^2$ satisfy (\ref{eqinhomog})   in $\Gamma(d)$ for some  $d\in (0,d_0)$ with $\epsilon_n\to 0$, and
\begin{equation}\label{eqBUlinfty}
  \|\varphi_n\|_{L^\infty\left(\Gamma(d) \right)}\leq1,\ \ \epsilon_n^2\|g_n\|_{L^\infty\left(\Gamma(d) \right)}\to 0.
\end{equation}

For each $p\in \Gamma$, passing to a subsequence if necessary, we have
\[\varphi_n\left(\epsilon_n b^{-1}(\epsilon_n s)z+\epsilon_n \zeta (\epsilon_n s), \epsilon_n s\right)
\to c_p \textbf{V}'(z)\ \textrm{in} \ C^1_{loc}(\mathbb{R}^N)\ \textrm{as}\ n\to \infty,
\]
for some $c_p\in \mathbb{R}$ and $\mathbf{V}=(V_1,V_2)$ as in Subsection \ref{subsectBerest}; where we have tacitly identified small geodesic balls on $\Gamma$ with  the corresponding neighborhood of the origin in $\mathbb{R}^{N-1}$ via the associated mapping.
\end{lem}
\begin{proof}
We first note that, owing to (\ref{eqinner})--(\ref{eqQ2}) and (\ref{eqzetaCoord}), we have
\begin{equation}\label{eqpot1}
U_1^2(x)=\epsilon^2b^2(s)V_1^2(z)+\epsilon^3\mathcal{O}\left(|z|^3+1\right), \ U_2^2(x)=\epsilon^2b^2(s)V_2^2(z)+\epsilon^3\mathcal{O}(e^{-cz}),
\end{equation}
$x\in \Gamma\left(2|\ln \epsilon|\epsilon \right)\cap \Omega_1$, together with the corresponding relations in $\Gamma\left(2|\ln \epsilon|\epsilon \right)\cap \Omega_2$, and
\begin{equation} \label{eqpot2}
U_1U_2(x)=\epsilon^2b^2(s)V_1V_2(z)+\epsilon^3\mathcal{O}\left(e^{-c|z|}\right), \ x \in \Gamma\left(2|\ln \epsilon|\epsilon \right),
\end{equation}
where, as always, the generic remainder satisfies (\ref{eqOss}).

Without loss of generality, we may assume that $p\in \Gamma$ is at the origin and that the outer  unit normal vector to $ \Gamma$ there is $(0,\cdots,1)$.
Hence, near the origin $\Gamma$ is the graph of a function $h:\mathbb{R}^{N-1}\to \mathbb{R}$ such that $h(0)=0$ and $\nabla h(0)=0$.

Let
\[
\Psi_n(z,y)=\Phi_n (z,\epsilon_n y),
\]
where $\Phi$ is as in (\ref{eqphiStretch}). Multiplying both sides of (\ref{eqinhomog}) by $b^{-2}(\epsilon_n y)\epsilon_n^2$, using (\ref{eqepsilonDef}), (\ref{eqL}), (\ref{eqLaplaceFermi41}), (\ref{eqpot1}), (\ref{eqpot2}), and dropping the subscript $n$ for notational convenience, we find that
\begin{equation}\label{eqlarge}
  \begin{split}
     \mathcal{M}(\Psi) -b^{-2}\Delta_{\tilde{\Gamma}_\epsilon} \Psi+\epsilon^2 \mathcal{O}(1+z^2)\Psi_{zz}+\epsilon\mathcal{O}\left(1+|z| \right)\nabla_{\tilde\Gamma_\epsilon}\Psi_z
\ \ \ \ \ \ \ \ \ \ \ \ \ \ \ \ \ \ \ \ \ \ \  \ \ \ \ \ \ \ &  \\
      \ \ \ \ \ \ \  +\epsilon\mathcal{O}\left(1+|z|\right)\nabla^2_{\tilde{\Gamma}_\epsilon} \Psi  +\epsilon \mathcal{O}\left(1\right)  \Psi_z+\epsilon \mathcal{O}\left(1\right)\nabla_{\tilde{\Gamma}_\epsilon} \Psi+\epsilon\mathcal{O}\left(|z|^3+1\right)\Psi=b^{-2} \epsilon^2G, &
  \end{split}
\end{equation}
in an $N$-dimensional ball of radius $c/\epsilon$, where
\begin{equation}\label{eqM}
\mathcal{M}=-\partial^2_z+\left(\begin{array}{cc}
                                  V_1^2 & 2V_1V_2 \\
                                  2V_1V_2 & V_2^2
                                \end{array}
 \right)
\end{equation}
is the linearized operator associated to (\ref{eqBerest}),  \[\tilde{\Gamma}_\epsilon=\left\{\left(\epsilon y,h(\epsilon y)\right),\ |y|<c/\epsilon \right\}
\ \textrm{and}\
G(z,y)=g(r,\epsilon y) \ (\textrm{recall}\  (\ref{eqzetaCoord})).
\]

We note that the Laplace-Beltrami operator of $\tilde{\Gamma}_\epsilon$ written in local coordinates has the form
\[
\Delta_{\tilde{\Gamma}_\epsilon}=a^0_{ij}(\epsilon y)\partial_{ij}+\epsilon b^0_{j}(\epsilon y)\partial_{j},
\]
with
\[
a^0_{ij}(\epsilon y)\to \delta_{ij}\ \textrm{uniformly over compacts as}\ \epsilon\to 0\ \textrm{and}\ b^0_{j}(\epsilon y)=\mathcal{O}(1),
\]
see for instance \cite[pg. 1517]{delAnnals}.
Hence, since
\[
|\Psi_n|\leq 1, \ \epsilon_n^2G_n\to 0\ \textrm{uniformly over compacts as}\ n\to \infty,
\]
we deduce from (\ref{eqlarge}) using standard $W^{2,p}$ elliptic estimates and the usual diagonal argument that
\[
\Psi_n\to \Psi_\infty \ \textrm{in}\ C^{1,\alpha}_{loc}(\mathbb{R}^N)\ \textrm{as}\ n\to \infty,
\]
where $\Psi_\infty$ is a classical bounded solution to the limit problem
\begin{equation}\label{pao13}
\mathcal{M}(\Psi)-b^{-2}(0)\Delta_{\mathbb{R}^{N-1}}\Psi=0\ \textrm{in}\ \mathbb{R}^N.
\end{equation}
Now, we conclude from \cite[Thm. 1]{sourdisJDE} (after a simple rescaling) that
\[
\Psi_{\infty}(z,y)=c\textbf{V}'(z)\ \textrm{for some constant}\ c,
\]
from which the assertion of the lemma follows at once.
\end{proof}

\subsection{The contradicting sequence}\label{subsecContraSeq}
In order to prove Theorem \ref{thmMain}, we will argue by contradiction. To this end, we suppose that there exists a sequence $\epsilon_n\to 0$ and $\varphi_n$, $g_n$ which satisfy (\ref{eqinhomog}) for $\epsilon=\epsilon_n$ such that
\begin{equation}\label{eqcontraNorm}
  \|\varphi_n\|_0=1\ \textrm{and}\ \epsilon_n\|g_n\|_1\to 0.
\end{equation}

\subsection{Profile for $\varphi_n$, $\nabla\varphi_n$ in the inner zone $\Gamma(M\epsilon_n)$ with $M\gg 1$ independent of $n\gg 1$} Since $\|\cdot\|_{L^\infty(\Omega)}\leq\|\cdot\|_i$, $i=0,1$, Lemma \ref{lemBU} is more than enough to imply the following corollary.
\begin{cor}\label{corin1}
  The functions $\varphi_n$ satisfy the assertion of Lemma \ref{lemBU}.
\end{cor}

\subsection{Exponential decay of $\varphi_{i,n}$ in $\Omega_j\setminus \Gamma(C\epsilon_n)$ with $i\neq j$}\label{subsecExpDecay}In this subsection, we will prove the following useful proposition (cf. the beginning of the proof of Proposition 4.1 in \cite{casteras} for the radially symmetric case).
\begin{pro}\label{proexpdec}
If $i\neq j$, then
\[|\varphi_{i,n}(x)|\leq C e^{-c\textrm{dist}(x,\Gamma)/\epsilon_n}+\mathcal{O}(\epsilon_n^\infty), \ x\in \Omega_j,\]
with constants $c,C>0$ independent of large $n$.
\end{pro}
\begin{proof}
We will provide the proof only for the case $i=2$ and $j=1$, since the remaining case can be handled analogously.

From (\ref{eqBerestAsympt}), (\ref{eqinner}), (\ref{eqb}), (\ref{eqzeta}) and (\ref{eqinnerremalg}), we obtain that
there exists a large constant $C>0$ such that
\begin{equation}\label{eqU1lower1}
U_1(r,s)\geq c \epsilon, \  C\epsilon \leq r \leq |\ln \epsilon|\epsilon;
\end{equation}
here and elsewhere $c/C$ denotes a small/large generic constant which is independent of large $n$ and whose value will decrease/increase from line to line as the proof progresses; abusing notation we will frequently drop the subscript $n$.
On the other side, recalling Definition \ref{defSep} and that $f(0,x)=0$ (keep in mind (\ref{eqfodd})), by Hopf's boundary point lemma on $\Gamma$ and $\partial \Omega$,
we deduce that \begin{equation}\label{eqhopf24}\partial_\mu w_1<0\ \textrm{on}\ \partial \Omega_1,\ \textrm{i.e.}\ w_1(x)\geq c \textrm{dist}(x,\partial\Omega_1),\end{equation}
where $\mu$ stands for the outer unit normal vector to $\partial \Omega_1$
 (see also (\ref{eqwHopf})). In turn, thanks to  (\ref{eqoutdef1}) and (\ref{eqU1lower1}),   we have
\begin{equation}\label{eqpot3}U_1(x)\geq c\epsilon\ \textrm{if} \ x\in \Omega_1 \ \textrm{and}\ \textrm{dist}(x,\partial \Omega_1)\geq C\epsilon.\end{equation}

Therefore, we see from  (\ref{eqb}), (\ref{eqoutdef1}), (\ref{eqL}), (\ref{eqinhomog}), (\ref{eqpot2}),  (\ref{eqcontraNorm}) (recalling the definition of the norms from (\ref{eqNorm0})-(\ref{eqnorm1})), (\ref{eqhopf24})  and (\ref{eqpot3})  that the second component of $\varphi$ satisfies
\begin{equation}\label{eqschrodinger}
-\epsilon^2\Delta \varphi_2+P(x)\varphi_2=Q(x), \ |\varphi_2|\leq 1,
\end{equation}
where
\begin{equation}\label{eqpot4}
c\leq P\leq C,\ |Q|\leq C\exp\{-c\textrm{dist}(x,\Gamma)/\epsilon\}+\mathcal{O}(\epsilon^\infty)\ \textrm{if} \ x\in \Omega_1 \ \textrm{and}\ \textrm{dist}(x,\partial \Omega_1)\geq C\epsilon.
\end{equation}

Then, a well known argument, see for instance \cite[Prop. 4.1]{sz24}, yields
\begin{equation}\label{eqexp1}
|\varphi_2|\leq C\exp\{-c\textrm{dist}(x,\partial \Omega_1)/\epsilon\}+\mathcal{O}(\epsilon^\infty)\ \textrm{if}\ x\in \Omega_1\ \textrm{and}\ \textrm{dist}(x,\partial \Omega_1)\geq C \epsilon.
\end{equation}

Unfortunately, the lower bound in the beginning of (\ref{eqpot4}) breaks down near $\partial \Omega_1 \setminus \Gamma$. More precisely, in light of (\ref{eqoutdef1}) and the fact that $w_1$ and $U_1$ vanish on $\partial\Omega_1\setminus \Gamma$, we have
\begin{equation}\label{eqS}
\epsilon^2P=U_1^2=w_1^2+\mathcal{O}(\epsilon)\left[\textrm{dist}(x,\partial \Omega_1\setminus \Gamma)\right]^2\ \textrm{in}\ \mathcal{S}=\left\{x\in \Omega_1\ :\ \textrm{dist}(x,\partial \Omega_1\setminus \Gamma)<d_1\right\},
\end{equation}
where $d_1$ is independent of $\epsilon$ and such that the corresponding mapping to $\tau$ (recall Subsection \ref{subsectfermi})
is well defined.
Nevertheless, based on (\ref{eqhopf24}) and (\ref{eqschrodinger}), we will be able to extend the validity of the bound (\ref{eqexp1}) up to $\partial \Omega_1 \setminus \Gamma$. It turns out that this task simplifies considerably if $f$ is identically equal to zero for $x\in \mathcal{S}$.
Henceforth, we will continue under this simplifying assumption   and refer to Remark \ref{remGenScrod} below for the general case.
By combining (\ref{eqschrodinger}),  (\ref{eqpot4}), (\ref{eqexp1})  and (\ref{eqS}), we find that $\varphi_2$ satisfies
\begin{equation}\label{eqSchrod2bi}\left\{\begin{array}{l}
-\epsilon^4 \Delta \varphi+w_1^2\varphi=\mathcal{O}(\epsilon)\left[\textrm{dist}(x,\partial\Omega_1\setminus \Gamma)\right]^2\varphi+\mathcal{O}(\epsilon^\infty)
\ \textrm{in}\ \mathcal{S};
\\
\ \varphi=\mathcal{O}(\epsilon^\infty)\ \textrm{on} \ \partial{\mathcal{S}}\cap \Omega,\
\varphi=0\ \textrm{on} \ \partial{\mathcal{S}}\cap \partial\Omega.
                                          \end{array}\right.
\end{equation}
In order to have homogeneous Dirichlet boundary conditions, we will add to $\varphi_2$ a harmonic correction such that
\[-\Delta h=0\ \textrm{in}\ \mathcal{S};\ h=-\varphi_2\ \textrm{on}\ \partial \mathcal{S}. \]
Clearly,  the maximum principle yields
\begin{equation}\label{eqhnorm}\|h\|_{L^\infty(\mathcal{S})}=\mathcal{O}(\epsilon^\infty).\end{equation}Then, letting
\begin{equation}\label{eqpsidef24yo}
\psi=\varphi_2+h,
\end{equation}
we obtain from (\ref{eqSchrod2bi}) that
\begin{equation}\label{eqschrod3c}
-\epsilon^4 \Delta \psi+w_1^2\psi=\mathcal{O}(\epsilon)\left[\textrm{dist}(x,\partial\Omega_1\setminus \Gamma)\right]^2\psi
+\mathcal{O}(\epsilon^\infty)
\ \textrm{in}\ \mathcal{S};
\ \psi=0\ \textrm{on} \ \partial{\mathcal{S}}.\end{equation}
Now, we will use a maximum principle type of argument in order to show that
\begin{equation}\label{eqpsiuniformtrans}
\psi=\mathcal{O}(\epsilon^\infty)\ \textrm{uniformly  on}\  \bar{\mathcal{S}}.\end{equation}
To this end, let us argue by contradiction and assume that there exists an $M>1$ (independent of $n$) such that
\begin{equation}\label{eqcontrapsimeat}
\|\psi\|_{L^\infty(\mathcal{S})}\geq c \epsilon^M,
\end{equation}
for some subsequence $\epsilon_n \to 0$.
  Without loss of generality, we may assume that
\begin{equation}\label{eqmaxpri24b}
\psi(x_n)=\|\psi\|_{L^\infty(\mathcal{S})}\ \textrm{for some}\ x_n\in \mathcal{S},\ \textrm{and therefore}\ \Delta \psi(x_n)\leq 0.
\end{equation}
We first observe that $x_n$ cannot get 'too close' to $\partial \mathcal{S}$.
Indeed, from (\ref{eqcontraNorm}) and (\ref{eqschrod3c}) we see that
\[
\Delta \psi=\mathcal{O}(\epsilon^{-4}) \ \textrm{uniformly on}\ \bar{\mathcal{S}};\ \psi=0\ \textrm{on}\ \partial \mathcal{S}.
\]
Then, by standard elliptic regularity theory up to the boundary (see also \cite[Lem. A.2]{bbh24}), we obtain that
\[
|\nabla \psi|=\mathcal{O}(\epsilon^{-4})  \ \textrm{uniformly on}\ \bar{\mathcal{S}}.
\]
So, in view of (\ref{eqcontrapsimeat}), (\ref{eqmaxpri24b}) and the fact that $\psi=0$ on $\partial \mathcal{S}$, we find that
\begin{equation}\label{eqviastrada}
\textrm{dist}(x_n,\partial \mathcal{S})\geq c \epsilon^{M+4}.
\end{equation}
Substituting $x=x_n$ in (\ref{eqschrod3c}), keeping in mind (\ref{eqhopf24}) and  (\ref{eqmaxpri24b}), gives
\[
\|\psi\|_{L^\infty(\mathcal{S})}\leq C\epsilon \|\psi\|_{L^\infty(\mathcal{S})}+\frac{\mathcal{O}(\epsilon^\infty)}{\left[\textrm{dist}(x_n,\partial \mathcal{S}\cap \partial\Omega)\right]^2}.
\]
On the other hand, the above relation via (\ref{eqviastrada}) readily yields
\[
\|\psi\|_{L^\infty(\mathcal{S})}=\mathcal{O}(\epsilon^\infty),
\]
which contradicts (\ref{eqcontrapsimeat}). Hence, we deduce that (\ref{eqpsiuniformtrans}) holds.
Consequently, by combining  (\ref{eqhnorm}), (\ref{eqpsidef24yo})  and (\ref{eqpsiuniformtrans}), we infer that
\[
\|\varphi_2 \|_{L^\infty(\mathcal{S})}=\mathcal{O}(\epsilon^\infty).
\]

The assertion of the proposition now follows at once from (\ref{eqcontraNorm}), (\ref{eqexp1}) and the above relation.
\end{proof}

\begin{rem}\label{remGenScrod}
In this remark, we will indicate how to deal with the general case in  Proposition \ref{proexpdec}  with $f$ nontrivial  near $\partial \Omega$.
We insist that everything up to (\ref{eqexp1}) is valid under this general assumption. However, now (\ref{eqSchrod2bi}) and in turn (\ref{eqschrod3c})   include a term of
the form $\mathcal{O}(\epsilon^4)\varphi$ and $\mathcal{O}(\epsilon^4)\psi$, respectively. From the latter relation, thanks to Lemma \ref{lemA1}  of Appendix \ref{ApendixA}, we obtain
\[
\|\psi\|_{L^\infty(\mathcal{S})}\leq C\epsilon \|\psi\|_{L^\infty(\mathcal{S})}+C\epsilon^2 \|\psi\|_{L^\infty(\mathcal{S})}+\mathcal{O}(\epsilon^\infty),
\]
i.e., $\psi$ is still transcendentally small in $\epsilon$ and we can conclude as before.
\end{rem}

\begin{rem}\label{remGradExpDec}
For future reference, we note that in light of Subsection \ref{subsecAnsatz}, we have
\begin{equation}\label{eqremexpdecay1}
  0<U_i(x)+\epsilon\left|\nabla U_i\right|\leq C\epsilon e^{-c\textrm{dist}(x,\Gamma)/\epsilon}+\mathcal{O}(\epsilon^\infty),\ x\in \Omega_j,\ i\neq j,
\end{equation}
and
\begin{equation}\label{eqremexpdecay2}
  0<U_1U_2(x)+\epsilon\left|\nabla (U_1U_2)\right|\leq C\epsilon^2e^{-c\textrm{dist}(x,\Gamma)/\epsilon}+\mathcal{O}(\epsilon^\infty),\ x\in \Omega.
\end{equation}
In turn, from (\ref{eqL}), (\ref{eqinhomog}), (\ref{eqcontraNorm}), Proposition \ref{proexpdec}, and standard interior and boundary elliptic estimates (see also
\cite[Lems. A.1, A.2]{bbh24}), we obtain
\[
|\nabla \varphi_i(x)|=\mathcal{O}(\epsilon^\infty)\ \textrm{on}\ \overline{\Omega_j}\setminus \Gamma(\delta),\ i\neq j,
\]
for any $\delta\in (0,d_0)$ independent of $n$ ($d_0$ as in Subsection \ref{subsectfermi}).
\end{rem}

\subsection{Convergence of the contradicting sequence in $C^{2,\beta}_{loc}(\Omega\setminus \Gamma)$}\label{subsecoutC1Star2fall}
The following proposition deals with the behaviour of $(\varphi_{1,n},\varphi_{2,n})$ in the outer region $\Omega\setminus \Gamma(\delta)$, $\delta\in (0,d_0)$ independent of $n$, as $n\to \infty$.\begin{pro}\label{proOutC2}
There exists a subsequence which we still denote by $(\varphi_{1,n},\varphi_{2,n})$ such that
\[
(\varphi_{1,n},\varphi_{2,n})\to(\mathds{1}_{\Omega_1}\varphi_{\infty},\mathds{1}_{\Omega_2}\varphi_{\infty})
\ \textrm{in}\ C_{loc}^{2,\beta}(\Omega \setminus \Gamma) \times C_{loc}^{2,\beta}(\Omega \setminus \Gamma) \ \textrm{as}\ n\to \infty,
\]
for any $\beta \in (0,1)$, where $\varphi_{\infty}\in C^2(\Omega\setminus \Gamma)\cap C(\bar{\Omega}\setminus \Gamma)$ satisfies
\begin{equation}\label{eqlimit}
-\Delta \varphi=f_u(w,x)\varphi \ \textrm{in}\ \Omega\setminus \Gamma; \ \varphi=0\ \textrm{on}\ \partial \Omega.
\end{equation}
\end{pro}
\begin{proof}
Let $m\in \mathbb{N}$ be such that $1/m\in(0,d_0)$ and independent of $n$.
By virtue of  (\ref{eqL}),
(\ref{eqinhomog}), (\ref{eqcontraNorm}), Proposition \ref{proexpdec}, (\ref{eqremexpdecay1}), (\ref{eqremexpdecay2}) and what follows in Remark \ref{remGradExpDec}, we have\[-\Delta \varphi_i-f_u(U_i,x)\varphi_i=\mathcal{O}_{C^1}(\epsilon^\alpha)\ \textrm{in}\ \Omega_i\setminus \Gamma(1/m);\ \varphi_i=0\ \textrm{on}\ \partial \Omega_i\setminus \Gamma, \]
\[-\Delta \varphi_j-f_u(U_j,x)\varphi_j=\mathcal{O}_{C^1}(\epsilon^\infty)\ \textrm{in}\ \Omega_i\setminus \Gamma(1/m);\ \varphi_j=0\ \textrm{on}\ \partial \Omega_i\setminus \Gamma,\ j\neq i, \]
where $\alpha \in (0,1)$ as in (\ref{eqweight}). Next, taking into account the assumptions in  Subsection \ref{subsecAnsatz} and since $f$ is at least twice continuously  differentiable, we obtain
\begin{equation}\label{eqiout}-\Delta \varphi_i-f_u(w_i,x)\varphi_i=\mathcal{O}_{C^1}(\epsilon^\alpha)\ \textrm{in}\ \Omega_i\setminus \Gamma(1/m);\ \varphi_i=0\ \textrm{on}\ \partial \Omega_i\setminus \Gamma, \end{equation}
\[-\Delta \varphi_j-f_u(0,x)\varphi_j=\mathcal{O}_{C^1}(\epsilon^\infty)\ \textrm{in}\ \Omega_i\setminus \Gamma(1/m);\ \varphi_j=0\ \textrm{on}\ \partial \Omega_i\setminus \Gamma,\ j\neq i. \]

By standard elliptic estimates, we infer that \[\varphi_1,\varphi_2\  \textrm{are bounded in}\ C^{2,\beta}\left(\Omega\setminus \Gamma\left(1/(m+1)\right)\right),\] uniformly with respect to $n\gg1$, for any $\beta \in (0,1)$.
A standard diagonal compactness argument then yields that as $n\to \infty$, possibly along a subsequence, $\varphi_1,\varphi_2$ have a limit in $C_{loc}^{2,\beta}\left(\Omega\setminus \Gamma\right)$ for any $\beta \in (0,1)$. The assertion of the proposition now follows readily by passing to the limit in (\ref{eqiout}) and recalling Proposition \ref{proexpdec}.  \end{proof}The above proposition yields merely that $\varphi_\infty\in C^{2+\beta}(\Omega \setminus \Gamma)$ for any $\beta\in (0,1)$. However, we will need to show that
each $\mathds{1}_{\Omega_i}\varphi_\infty$ is smooth up to $\Gamma$. To this end, we first make the following easy observation.
\begin{lem}\label{lemtangDiffGhost}
We have
\[
|\varphi_\infty|+|\nabla_\Gamma\varphi_\infty|+|\nabla^2_\Gamma\varphi_\infty|\leq C \  \textrm{in}\  \Gamma(d_0)\setminus \Gamma.
\]
\end{lem}
\begin{proof}
Let $x\in \Gamma(d_0)\setminus \Gamma$. We know from (\ref{eqcontraNorm}) that
\[
|\varphi_{i,n}|+|\nabla_\Gamma\varphi_{i,n}|+|\nabla^2_\Gamma\varphi_{i,n}|\leq C \  \textrm{in}\ B_x\left(\textrm{dist}(x,\Gamma)/2\right).
\]
The assertion of the lemma now follows at once thanks to Proposition \ref{proOutC2} by letting $n\to \infty$ in the above relation.
\end{proof}\subsection{Connecting the outer to the inner zone in $C^1$ normal to $\Gamma$}\label{subsecexchange}In this subsection, we will 'push' the convergence in Proposition \ref{proOutC2}  through the intermediate zone $\Gamma(d_0)\setminus \Gamma(M\epsilon)$ with $M\gg 1$ independent of $\epsilon$. In this regard, we have the following.\begin{pro}\label{proexchangelemma}
If $x\in \Omega_i$ and $ \textrm{dist}(x,\Gamma) < 2d_0$,  then
\begin{equation}\label{eqC2yakoubian}
\left|\varphi_{i}(x)-\varphi_\infty(x)\right|+\epsilon \left|\left(\varphi_{i}-\varphi_\infty\right)_r\right|=\mathcal{O}(1)e^{-c\textrm{dist}(x,\Gamma)/\epsilon}+o(1),
\end{equation}
where  the  Landau symbols satisfy
\[
\left|\mathcal{O}(1)\right|\leq C\ \textrm{and}\ o(1)\to 0,  \ \textrm{uniformly in}\ x,\ \textrm{as}\ n\to \infty.
\]
\end{pro}\begin{proof}For convenience, we will only provide the proof in the case $i=1$ (the case $i=2$ can be treated completely analogously).

Let
\[
\psi=\varphi_1-\varphi_\infty,\ x\in \Omega_1.
\]
In light of (\ref{eqL}), (\ref{eqinhomog}), (\ref{eqcontraNorm}), (\ref{eqremexpdecay1}) and (\ref{eqremexpdecay2}), we have
\begin{equation}\label{eqdiffbmg1}
-\Delta \psi= \mathcal{O}(\epsilon^{-2})e^{-c\textrm{dist}(x,\Gamma)/\epsilon}+ f_u(U_1,x)\varphi_1-f_u(w_1,x)\varphi_\infty+g_1,\ x\in \Omega_1.
\end{equation}

According to the notation from Subsection \ref{subsectfermi}, and (\ref{eqLaplaceFermi1}), we can write
\begin{equation}\label{eqdiffbmg2}
\Delta \psi= J^{-1}\frac{\partial}{\partial r}\left(J\frac{\partial \psi}{\partial r}\right)+\Delta_{\Gamma_r}\psi,
\end{equation}
where
\begin{equation}\label{eqjacDef}
J(r,s)=\textrm{det}\frac{\partial \tau^{-1}(r,s)}{\partial (r,s)}
\end{equation}
denotes the Jacobian of the mapping $\tau^{-1}$ (see also \cite[pg. 1380]{xinfu}). For future reference, we note that \begin{equation}\label{eqJ=1}J(0,s)=1.\end{equation}

By Proposition \ref{proOutC2},  we obtain
\begin{equation}\label{eqdiffbmg3}
  \|f_u(U_1,x)\varphi_1-f_u(w_1,x)\varphi_\infty \|_{L^1(0,2d_0)}=o(1),\ \textrm{uniformly for}\ s\in \Gamma,\ \textrm{as}\ n\to \infty.
\end{equation}
In fact, we need to split the above integral in $(0,\delta)$ and $(\delta, 2d_0)$, first choose $\delta$ sufficiently small  but indepent of $n$ to control the first integral and then let $n\to \infty$ (recall also that
 $U_1-w_1=\mathcal{O}(\epsilon)$ from Subsection \ref{subsecAnsatz}, $|\varphi_1|,\ |\varphi_2|\leq 1$ and that $f_u$ is smooth).
Similarly, from (\ref{eqLaplaceFermi2}), (\ref{eqLaplaceFermi3}),  Proposition \ref{proOutC2} and Lemma \ref{lemtangDiffGhost}, we find
\begin{equation}\label{eqdiffbmg4}
  \|\Delta_{\Gamma_r}\psi\|_{L^1(0,2d_0)}=o(1),\ \textrm{uniformly for}\ s\in \Gamma,\ \textrm{as}\ n\to \infty.
\end{equation}
Moreover, owing to (\ref{eqnorm1}) and (\ref{eqcontraNorm}), we get
\begin{equation}\label{eqdiffbmg5}
\|g_1\|_{L^1(0,2d_0)}=o(1)\int_{0}^{2d_0}\frac{\epsilon^{-1}}{1+(\epsilon^{-1}r)^{1+\alpha}}dr=o(1)\int_{0}^{2d_0/\epsilon}\frac{1}{1+r^{1+\alpha}}dr=o(1),
\end{equation}uniformly for $s\in \Gamma$, as $n\to \infty$.
Hence, testing  (\ref{eqdiffbmg1}) with $J$ in $(r,2d_0)$, and using (\ref{eqdiffbmg2}), (\ref{eqdiffbmg3}), (\ref{eqdiffbmg4}) and (\ref{eqdiffbmg5}), we infer that
\[
  J(r,s)\frac{\partial \psi(r,s)}{\partial r}-J(2d_0,s)\frac{\partial \psi(2d_0,s)}{\partial r}=\mathcal{O}(\epsilon^{-1})e^{-cr/\epsilon}+o(1),
\]
uniformly for $s\in \Gamma$, as $n\to \infty$. Now,   thanks to Proposition \ref{proOutC2}, we arrive at
\[
\frac{\partial \psi(r,s)}{\partial r}=\mathcal{O}(\epsilon^{-1})e^{-cr/\epsilon}+o(1),
\]
uniformly for $s\in \Gamma$, as $n\to \infty$. This proves the $C^1$ part of (\ref{eqC2yakoubian}). The $C^0$ part now follows at once by integrating the above relation in $(r,2d_0)$ and making again use of Proposition \ref{proOutC2}.\end{proof}

\subsection{Smoothness of $\mathds{1}_{\Omega_1}\varphi_\infty-\mathds{1}_{\Omega_2}\varphi_\infty$ across $\Gamma$ and establishing $\varphi_{\infty}\equiv 0$}\label{subsecreflection} In this subsection, working in the vicinity of $\Gamma$,  we will first derive  'approximate reflection laws'  between $\mathds{1}_{\Omega_1}\varphi_{\infty}$ and $\mathds{1}_{\Omega_2}\varphi_{\infty}$ in the direction orthogonal to $\Gamma$. First, these will imply that each $\mathds{1}_{\Omega_i}\varphi_\infty$ can be extended smoothly up to $\Gamma$. Then, we will show that  the difference $\mathds{1}_{\Omega_1}\varphi_\infty-\mathds{1}_{\Omega_2}\varphi_\infty$ is a classical solution to   (\ref{ewlinearFul}), and thus has to be identically equal to zero by the nondegeneracy assumption from Definition \ref{defNonDeg}.  To this end, we will adapt arguments from the proof of \cite[Prop. 4.1]{casteras}.

The main effort in this subsection will be placed in establishing the following result.
\begin{pro}\label{proreflectt}There exists a constant $C>0$ such that
\begin{equation}\label{eqrefl1}
\left|J(-\delta_2,s)\partial_r\varphi_\infty(-\delta_2,s)  + J(\delta_1,s)\partial_r\varphi_\infty(\delta_1,s)\right|\leq C(\delta_1+\delta_2)
\end{equation}
and
\begin{equation}\label{eqrefl2}\begin{split}
                                  \left|J(-\delta_2,s)\left(\varphi_\infty(-\delta_2,s)-\partial_r\varphi_\infty(-\delta_2,s)\delta_2\right)+J(\delta_1,s)\left(\varphi_\infty(\delta_1,s)-\partial_r\varphi_\infty(\delta_1,s)\delta_1\right)\right| \ \                              \\
                                    \leq C(\delta_1+\delta_2),
                               \end{split}
\end{equation}
for any $\delta_1,\delta_2\in (0,2d_0)$ and $s\in \Gamma$, where $J$ is as in (\ref{eqjacDef}).
\end{pro}
\begin{proof}In light of (\ref{eqL}), (\ref{eqNorm0}), (\ref{eqcontraNorm}) and (\ref{eqdiffbmg2}), we can write (\ref{eqinhomog}) as
\begin{equation}\label{eqjac1}
  -J^{-1}\frac{\partial}{\partial r}\left(\begin{array}{c}
                                            J\partial_r \varphi_1 \\
                                            J\partial_r \varphi_2
                                          \end{array}
  \right)+\epsilon^{-4}\left(\begin{array}{cc}
                 U_2^2  & 2U_1U_2 \\
                  2U_1U_2 & U_1^2
                \end{array}
   \right)\left(\begin{array}{c}
                                              \varphi_1 \\
                                              \varphi_2
                                          \end{array}
  \right)+\mathcal{O}(1)=\left(\begin{array}{c}
                                              g_1 \\
                                              g_2
                                          \end{array}
  \right),
\end{equation}
uniformly in $\Gamma(2d_0)$, as $n\to \infty$.

By the assumptions from Subsections \ref{subsubΙΝAns}  and \ref{subsubOutAns}, recalling (\ref{eqzetaCoord}), for $x\in\Gamma(2d_0)$ we find
\begin{equation}\label{eqeaster1}\begin{array}{lll}
                                   U_2^2(x)&=&\epsilon^2b^2(s)V_2^2(z)+2\epsilon^3b H(0,s) V_2W_2(z)+R_{11}(z,s), \\
                                   U_1^2(x)&=&\epsilon^2b^2(s)V_1^2(z)+2\epsilon^3b H(0,s) V_1W_1(z)+R_{22}(z,s), \\
                                   U_1U_2(x)&=&\epsilon^2b^2(s)V_1V_2(z)+\epsilon^3b H(0,s) (V_1W_2+V_2W_1)+R_{12}(z,s),
                                 \end{array}
\end{equation}
where
\begin{equation}\label{eqeaster2}
\left|R_{22}(z,s)\right|\leq\left\{\begin{array}{ll}
                                     C\epsilon^4(z^4+1) & \textrm{if}\ x\in \Gamma(2d_0)\cap \Omega_1, \\
                                     C\epsilon^4e^{cz}+\mathcal{O}(\epsilon^\infty) & \textrm{if}\ x\in \Gamma(2d_0)\cap \Omega_2,
                                   \end{array}
\right.\ \left|R_{11}(z,s)\right|\leq C \left|R_{22}(-z,s)\right|,
\end{equation}
\begin{equation}\label{eqeaster3}
\left|R_{12}(z,s)\right|\leq
                                     C\epsilon^4e^{-c|z|}+\mathcal{O}(\epsilon^\infty) \ \textrm{if}\ x\in \Gamma(2d_0).
                        \end{equation}

We consider the perturbed linear operator
\begin{equation}\label{eqMtilde}
\tilde{\mathcal{M}}=\mathcal{M}+2\epsilon b^{-1}(s)H(0,s) \left(\begin{array}{cc}
                                                           V_2W_2 & V_1W_2+V_2W_1 \\
                                                           V_1W_2+V_2W_1 & V_1W_1
                                                         \end{array}
\right)I,
\end{equation}
where $\mathcal{M}$ is as in (\ref{eqM}).
By the translation and scaling invariance of (\ref{eqBerest}), described at the end of Subsection \ref{subsectBerest}, we see that the pairs
\[
(V_1',V_2')\ \textrm{and}\ (zV_1'+V_1,zV_2'+V_2)
\]
belong in the kernel of $\mathcal{M}$. It is therefore natural to seek corresponding refined elements in the approximate kernel of $\tilde{\mathcal{M}}$ in the form
\begin{equation}\label{eqkernel}
\begin{array}{c}
  \left(\begin{array}{c}
        \Phi_1 \\
        \Phi_2
      \end{array}
 \right)=\left(\begin{array}{c}
        V_1' \\
        V_2'
      \end{array}
 \right)+2\epsilon b^{-1}(s)H(0,s) \left(\begin{array}{c}
        \hat{\Phi}_1 \\
        \hat{\Phi}_2
      \end{array}
 \right), \\
   \\
  \left(\begin{array}{c}
        \Psi_1 \\
        \Psi_2
      \end{array}
 \right)=\left(\begin{array}{c}
        zV_1'+V_1 \\
        zV_2'+V_2
      \end{array}
 \right)+2\epsilon b^{-1}(s)H(0,s) \left(\begin{array}{c}
        \hat{\Psi}_1 \\
        \hat{\Psi}_2
      \end{array}
 \right).
\end{array}
\end{equation}
 We make the choices
\begin{equation}\label{eqapp3}
 \mathcal{M}\left(\begin{array}{c}
        \hat{\Phi}_1 \\
        \hat{\Phi}_2
      \end{array}
 \right)=-  \left(\begin{array}{cc}
                                                           V_2W_2 & V_1W_2+V_2W_1 \\
                                                           V_1W_2+V_2W_1 & V_1W_1
                                                         \end{array}
\right)\left(\begin{array}{c}
        V_1' \\
        V_2'
      \end{array}
 \right),\ z\in \mathbb{R},
\end{equation}
and
\begin{equation}\label{eqapp4}
\mathcal{M}\left(\begin{array}{c}
        \hat{\Psi}_1 \\
        \hat{\Psi}_2
      \end{array}
 \right)=-  \left(\begin{array}{cc}
                                                           V_2W_2 & V_1W_2+V_2W_1 \\
                                                           V_1W_2+V_2W_1 & V_1W_1
                                                         \end{array}
\right)\left(\begin{array}{c}
        zV_1'+V_1 \\
        zV_2'+V_2
      \end{array}
 \right),\ z\in \mathbb{R}.
\end{equation}
We note that the righthand sides of (\ref{eqapp3}) and (\ref{eqapp4}) converge to zero super-exponentially fast as $|z|\to \infty$. In fact, thanks to the symmetry properties of $\textbf{V}$ and $\textbf{W}$ (recall the discussion leading to (\ref{eqBerestAsympt}), and (\ref{eqWdiaz2})),  we find that the righthand side of (\ref{eqapp3}) enjoys the same symmetry property as $\textbf{V}$  while that of (\ref{eqapp4}) satisfies the same symmetry property as $\textbf{W}$. Hence, by \cite[Lem. 2.4]{aftalionSour}, we infer that (\ref{eqapp3})   admits a solution such that
\begin{equation}\label{eqappass1}
\left(\hat{\Phi}_1(-z),\hat{\Phi}_2(-z) \right)=\left(\hat{\Phi}_2(z),\hat{\Phi}_1(z) \right),
\
  \left(\hat{\Phi}_1(z),\hat{\Phi}_2(z) \right)- (a,0)=\mathcal{O}(e^{-Dz}), \   z>0,
\end{equation}
for some $a\in \mathbb{R}$ and all $D>0$.
Similarly, system (\ref{eqapp4}) has a solution such that
\begin{equation}\label{eqappass2}
\left(\hat{\Psi}_1(-z),\hat{\Psi}_2(-z) \right)=-\left(\hat{\Psi}_2(z),\hat{\Psi}_1(z) \right),
\
  \left(\hat{\Psi}_1(z),\hat{\Psi}_2(z) \right)- (bz,0)=\mathcal{O}(e^{-Dz}), \   z>0,
\end{equation}
for some $b\in \mathbb{R}$ and all $D>0$. We point out that the estimates in  the above two relations can be differentiated in the natural way
and that $\hat{\Phi}_i$, $\hat{\Psi}_i$ are independent of both $\epsilon$ and $s$.
It follows readily from (\ref{eqMtilde})-(\ref{eqappass2}) and the asymptotic behaviour of the functions involved that
\begin{equation}\label{eqMtilderemain}
\begin{array}{rcl}
 \tilde{\mathcal{M}}\left(\begin{array}{c}
        \Theta_1 \\
         \Theta_2
      \end{array}
 \right) & = & 4\epsilon^2 b^{-2}(s) H^2(0,s) \left(\begin{array}{cc}
                                                           V_2W_2 & V_1W_2+V_2W_1 \\
                                                           V_1W_2+V_2W_1 & V_1W_1
                                                         \end{array}
\right)\left(\begin{array}{c}
        \hat{\Theta}_1 \\
        \hat{\Theta}_2
      \end{array}
 \right) \\
   &  &  \\
   &=  &  \epsilon^2  \left(\begin{array}{c}
                                                           \mathcal{O}(e^{-cz^2}) \\
                                                           \mathcal{O}(e^{-cz^2})
                                                         \end{array}
\right),\ z\in \mathbb{R},
\   \left(\begin{array}{c}
          \Theta_1 \\
          \Theta_2
        \end{array}
\right)=\left(\begin{array}{c}
                {\Phi}_1 \\
                {\Phi}_2
              \end{array}
\right)\ \textrm{or}\ \left(\begin{array}{c}
                              {\Psi}_1 \\
                              {\Psi}_2
                            \end{array}
\right).
\end{array}
\end{equation}

We let $\delta_1,\ \delta_2\in (0,2d_0)$ and $s\in \Gamma$
be arbitrary but independent of $n$ and denote
\begin{equation}\label{eqz1z2}
  z_-(s)=\epsilon^{-1}b(s) \left(-\delta_2 - \epsilon \zeta(s)\right),\ z_+(s)=\epsilon^{-1}b(s) \left(\delta_1 - \epsilon \zeta(s)\right).
\end{equation}

We will need the following preliminary observation which is valid for  $\phi\in C^2(-2d_0,2d_0)$ and $\Theta_i$ as in (\ref{eqMtilderemain}). Integrating by parts, and taking into account (\ref{eqphiStretch}) and what follows, yields
\begin{equation}\label{eqJacparts}\begin{array}{l}
                                    -\int_{-\delta_2}^{\delta_1}  \frac{\partial}{\partial r}\left(J\frac{\partial}{\partial r} \phi(r)\right)\Theta_i(z)dr= \\ \\
                                    -\epsilon^{-2}b^2\int_{-\delta_2}^{\delta_1}
\frac{\partial}{\partial z}\left[J\left(\epsilon \zeta(s)+\epsilon b^{-1}(s)z,s\right)\frac{\partial}{\partial z}\Theta_i(z) \right]\phi(r)dr
 \\ \\
                                    +\epsilon^{-1}b\left(J(\delta_1,s)\phi(\delta_1)\partial_z\Theta_i(z_+)-J(-\delta_2,s)\phi(-\delta_2)\partial_z\Theta_i(z_-) \right)
 \\ \\
                                    -\left(J(\delta_1,s)\partial_r\phi(\delta_1)\Theta_i(z_+)-J(-\delta_2,s)\partial_r\phi(-\delta_2)\Theta_i(z_-) \right).
                                  \end{array}
\end{equation}

Testing (\ref{eqjac1}) by $J(r,s)\left(\Theta_1(z),\Theta_2(z)\right)$, and making use of (\ref{eqeaster1}), (\ref{eqMtilde}), (\ref{eqz1z2}) and (\ref{eqJacparts}), we get
\begin{equation}\label{eqhuge}
\begin{array}{c}
  b^2(s)\epsilon^{-2}\left<J(r,s)\tilde{\mathcal{M}}\left(\begin{array}{c}
                \Theta_1(z) \\
                \Theta_2(z)
              \end{array}
\right), \left(\begin{array}{c}
                \varphi_1(r,s) \\
                \varphi_2(r,s)
              \end{array}
\right)\right>_{L^2\times L^2(-\delta_2,\delta_1)}
 \\
   \\
  -\epsilon^{-1}b\int_{-\delta_2}^{\delta_1}J_r\left(\epsilon \zeta(s)+\epsilon b^{-1}(s)z,s\right)
\left(\frac{\partial}{\partial z}\Theta_1(z)\varphi_1(r,s)+\frac{\partial}{\partial z}\Theta_2(z)\varphi_2(r,s) \right)dr
 \\
   \\
  +\epsilon^{-1}b\sum_{i=1}^{2}\left(J(\delta_1,s)\varphi_i(\delta_1,s)\partial_z\Theta_i(z_+)-J(-\delta_2,s)\varphi_i(-\delta_2,s)\partial_z\Theta_i(z_-) \right)
 \\
   \\
                                    -\sum_{i=1}^{2}\left(J(\delta_1,s)\partial_r\varphi_i(\delta_1,s)\Theta_i(z_+)-J(-\delta_2,s)\partial_r\varphi_i(-\delta_2,s)\Theta_i(z_-) \right) \\
   \\
  + \epsilon^{-4}\left<J\left(\begin{array}{cc}
                                                           R_{11}(z,s) & R_{12}(z,s) \\
                                                           R_{12}(z,s) & R_{22}(z,s)
                                                         \end{array}
\right)\left(\begin{array}{c}
                \Theta_1(z) \\
                \Theta_2(z)
              \end{array}
\right), \left(\begin{array}{c}
                \varphi_1(r,s) \\
                \varphi_2(r,s)
              \end{array}
\right)\right>_{L^2\times L^2(-\delta_2,\delta_1)}
 \\
   \\
  +\mathcal{O}(1)\int_{-\delta_2}^{\delta_1}\left(|\Theta_1(z)|+|\Theta_2(z)| \right)dr
 \\
   \\
  =\int_{-\delta_2}^{\delta_1}\left(g_1(r,s)\Theta_1(z)+g_2(r,s)\Theta_2(z) \right)J(r,s)dr.
\end{array}
\end{equation}

By (\ref{eqcontraNorm}) and (\ref{eqMtilderemain}), we find that the term in the first line of the above relation satisfies
\begin{equation}\label{eqline1}\begin{array}{rcl}
                                  b^2(s)\epsilon^{-2}\left|\left<J(r,s)\tilde{\mathcal{M}}\left(\begin{array}{c}
                \Theta_1(z) \\
                \Theta_2(z)
              \end{array}
\right), \left(\begin{array}{c}
                \varphi_1(r,s) \\
                \varphi_2(r,s)
              \end{array}
\right)\right>_{L^2\times L^2(-\delta_2,\delta_1)}\right| & \leq &C \int_{-\delta_2}^{\delta_1}e^{-c|z|}dr \\
                                  &  &  \\
                                  & \stackrel{(\ref{eqzetaCoord})}{\leq} & C \int_{-\delta_2}^{\delta_1}e^{-c|r|/\epsilon}dr \\
                                  &  &  \\
                                  & \leq & C\epsilon,
                               \end{array}
\end{equation}
where throughout this subsection, unless specified otherwise, the generic constants $c,C>0$ are independent of $n$, $\delta_1$, $\delta_2$ and $s\in \Gamma$. Concerning the fifth line in (\ref{eqhuge}), thanks to (\ref{eqcontraNorm}),  (\ref{eqeaster2}), (\ref{eqeaster3}), (\ref{eqappass1}) and (\ref{eqappass2}), working as above, we have
\begin{equation}\label{eqline2}
\epsilon^{-4}\left|\left<J\left(\begin{array}{cc}
                                                           R_{11}(z,s) & R_{12}(z,s) \\
                                                           R_{12}(z,s) & R_{22}(z,s)
                                                         \end{array}
\right)\left(\begin{array}{c}
                \Theta_1(z) \\
                \Theta_2(z)
              \end{array}
\right), \left(\begin{array}{c}
                \varphi_1(r,s) \\
                \varphi_2(r,s)
              \end{array}
\right)\right>_{L^2\times L^2(-\delta_2,\delta_1)}\right|\leq C\epsilon.
\end{equation}
We will estimate the remaining terms in (\ref{eqhuge}) separately for $(\Theta_1,\Theta_2)=(\Phi_1,\Phi_2)$ and
$(\Theta_1,\Theta_2)=({\Psi}_1,{\Psi}_2)$.

\underline{\textbf{$(\Theta_1,\Theta_2)=(\Phi_1,\Phi_2)$}} Concerning the second line in (\ref{eqhuge}) with $(\Theta_1,\Theta_2)=(\Phi_1,\Phi_2)$, in light of (\ref{eqkernel}) and (\ref{eqappass1}), arguing as in (\ref{eqline1}), we can write
\[
  \begin{split}
     \int_{-\delta_2}^{\delta_1}J_r\left(\epsilon \zeta(s)+\epsilon b^{-1}(s)z,s\right)
\left(\Phi_1'(z)\varphi_1(r,s)+\Phi_2'(z)\varphi_2(r,s) \right)dr=\ \ \ \ \ \ \ \ \ \ \ \ \  \ \ \ \ \ \ \ \ \ \ \ \ \ \ \ \ \ \ \ \ \ \ \ \ \ \ \ \ \ \ \ \   \\
\epsilon b^{-1}(s)          \int_{z_-(s)}^{z_+(s)}J_r\left(\epsilon \zeta(s)+\epsilon b^{-1}(s)z,s\right)
\left[\sum_{i=1}^{2}V_i''(z)\varphi_i(\epsilon \zeta+\epsilon b^{-1}z,s)\right]dz+\mathcal{O}(\epsilon^2).\ \ \ \ \ \ \ \ \ \ \ \ \ \ \ \ \ \ \ \  \ \ \ \
 \end{split}
\]
Using once more that $V_i''$ decays super-exponentially fast to zero as $|z|\to \infty$ and $|\varphi_i|\leq 1$, $i=1,2$, by virtue of Lemma \ref{lemBU}   and Lebesgue's dominated convergence theorem, for each $s\in \Gamma$, there exists a subsequence $\epsilon_{n,s}$ of $\epsilon_n$ such that
\[\begin{split}
                                                \int_{z_-}^{z_+}J_r\left(\epsilon_{n,s} \zeta+\epsilon_{n,s} b^{-1}z,s\right)
\left[\sum_{i=1}^{2}V_i''(z)\varphi_i(\epsilon_{n,s} \zeta+\epsilon_{n,s} b^{-1}z,s)\right]dz\\
                                                  \to c_sJ_r(0,s)\int_{-\infty}^{\infty}\textbf{V}''\cdot \textbf{V}'dz=0.
                                             \end{split}
\]
Hence, by combining the above two relations, we arrive at
\begin{equation}\label{eqline3}
\epsilon_{n,s}^{-1}b  \int_{-\delta_2}^{\delta_1}J_r\left(\epsilon_{n,s} \zeta+\epsilon_{n,s} b^{-1}z,s\right)
\left(\Phi_1'(z)\varphi_1(r,s)+\Phi_2'(z)\varphi_2(r,s) \right)dr\to 0\ \textrm{as}\ \epsilon_{n,s}\to 0.
\end{equation}
Then, owing to (\ref{eqBerestAsympt}), Proposition \ref{proOutC2}, (\ref{eqkernel}) and (\ref{eqappass1}), along the original sequence $\epsilon_n$, the terms in the third and fourth line of (\ref{eqhuge}) satisfy
\begin{equation}\label{eqline4}
\begin{array}{c}
  \epsilon^{-1}b\sum_{i=1}^{2}\left(J(\delta_1,s)\varphi_i(\delta_1,s)\Phi_i'(z_+)-J(-\delta_2,s)\varphi_i(-\delta_2,s)\Phi_i'(z_-) \right) \\
   \\
  -\sum_{i=1}^{2}\left(J(\delta_1,s)\partial_r\varphi_i(\delta_1,s)\Phi_i(z_+)-J(-\delta_2,s)\partial_r\varphi_i(-\delta_2,s)\Phi_i(z_-) \right) \\
   \\
  \to -A J(\delta_1,s)\partial_r\varphi_\infty(\delta_1,s)-A J(-\delta_2,s)\partial_r\varphi_\infty(-\delta_2,s).
\end{array}
\end{equation}
For the term in the sixth line of (\ref{eqhuge}), we plainly note that
\begin{equation}\label{eqline5}
\left|\mathcal{O}(1)\int_{-\delta_2}^{\delta_1}\left(|\Phi_1(z)|+|\Phi_2(z)| \right)dr\right|\leq C(\delta_1+\delta_2).
\end{equation}
The  last term in (\ref{eqhuge}) can be bounded as in (\ref{eqdiffbmg5}), i.e.,
\begin{equation}\label{eqline6}
\int_{-\delta_2}^{\delta_1}\left(g_1(r,s)\Phi_1(z)+g_2(r,s)\Phi_2(z) \right)J(r,s)dr\to 0\ \textrm{as}\ \epsilon_n\to 0.
\end{equation}
The desired estimate (\ref{eqrefl1}) now follows by letting $\epsilon_{n,s}\to 0$ in (\ref{eqhuge}), making use of (\ref{eqline1})-(\ref{eqline6}) (recall that $A>0$).

\underline{\textbf{$(\Theta_1,\Theta_2)=(\Psi_1,\Psi_2)$}} In this case, it turns out that less involved  estimates are needed in order to conclude. Indeed, for the corresponding second line in (\ref{eqhuge}) we just observe that
\[
  \epsilon^{-1}b  \left| \int_{-\delta_2}^{\delta_1}J_r\left(\epsilon \zeta+\epsilon b^{-1}z,s\right)
\left(\Psi_1'(z)\varphi_1(r,s)+\Psi_2'(z)\varphi_2(r,s) \right)dr \right|\leq C\epsilon^{-1}(\delta_1+\delta_2).
\]
Next, by (\ref{eqBerestAsympt}), Proposition \ref{proOutC2}, (\ref{eqkernel}) and (\ref{eqappass2}), we deduce that
\[
\begin{split}
\sum_{i=1}^{2}\left(   J(\delta_1,s)\varphi_i(\delta_1,s)\Psi_i'(z_+)-J(-\delta_2,s)\varphi_i(-\delta_2,s)\Psi_i'(z_-)\right) \to \ \ \ \ \ \ \ \ \ \ \ \ \ \ \ \ \ \ \ \ \ \ \ \ \ \ \ \ \ \  \   \\
  \ \ \ \ \ \ \ \ \ \ \ \ \ \ \ \ \ \ \   2AJ(\delta_1,s)\varphi_\infty(\delta_1,s)+2AJ(-\delta_2,s)\varphi_\infty(-\delta_2,s)\ \ \
\end{split}\]
and
       \[\begin{split}
            -\epsilon\sum_{i=1}^{2}\left(J(\delta_1,s)\partial_r\varphi_i(\delta_1,s)\Psi_i(z_+)-J(-\delta_2,s)\partial_r\varphi_i(-\delta_2,s)\Psi_i(z_-) \right)\to \ \ \ \ \ \ \ \ \ \ \ \ \ \ \ \  \  \\
\ \              -J(\delta_1,s)\partial_r\varphi_\infty(\delta_1,s)2Ab(s)\delta_1
-J(-\delta_2,s)\partial_r\varphi_\infty(-\delta_2,s)2Ab(s)\delta_2.
         \end{split}
\]
Lastly,  as in (\ref{eqline5}) and (\ref{eqline6}), we obtain
\[
\left|\mathcal{O}(1)\int_{-\delta_2}^{\delta_1}\left(|\Psi_1(z)|+|\Psi_2(z)| \right)dr\right|\leq C(\delta_1+\delta_2)\epsilon^{-1},
\]
and
\begin{equation}\label{eqline6}
\int_{-\delta_2}^{\delta_1}\left(g_1(r,s)\Psi_1(z)+g_2(r,s)\Psi_2(z) \right)J(r,s)dr=o(1)\epsilon^{-1}\ \textrm{as}\ \epsilon_n\to 0.
\end{equation}
In light of the above, the sought after estimate (\ref{eqrefl2}) follows readily by multiplying (\ref{eqhuge}) with $\epsilon_n$ and then letting $n\to \infty$.\end{proof}

As a first consequence of Proposition \ref{proreflectt}, we have the following.

\begin{cor}\label{corgradientBound}
There exists a constant $C>0$ such that
\begin{equation}\label{eqgradientboundInf}
|\nabla \varphi_\infty|\leq C\ \textrm{in}\ \Omega\setminus \Gamma,
\end{equation}
and therefore each $\mathds{1}_{\Omega_i}\varphi_\infty$ can be extended continuously up to $\Gamma$. In fact, we have
\begin{equation}\label{eqc0}\mathds{1}_{\Omega_2}\varphi_\infty=-\mathds{1}_{\Omega_1}\varphi_\infty \ \textrm{on}\ \Gamma. \end{equation}
\end{cor}

\begin{proof}We will provide the proof of (\ref{eqgradientboundInf}) only in $\Omega_1$. From (\ref{eqrefl1}), by fixing a $\delta_2\in (0,2d_0)$ and letting $r=\delta_1\in (0,2d_0)$ free, we obtain
\[\left|\partial_r\varphi_\infty(r,s)\right|\leq C
\]
for some constant $C>0$ that is independent of $(r,s)$ (recall also (\ref{eqJ=1})).
Moreover, we already know from Lemma \ref{lemtangDiffGhost} that \[\left|\nabla_\Gamma\varphi_\infty (r,s)\right|\leq C.\] The gradient estimate (\ref{eqgradientboundInf}) now follows at once.

Armed with  (\ref{eqgradientboundInf}), we can extend continuously each $\mathds{1}_{\Omega_i}\varphi_\infty$ up to $\Gamma$ thanks to the Arzela-Ascoli theorem.
Letting $\delta_1\to 0$ and $\delta_2\to 0$ in (\ref{eqrefl2}), making again use of (\ref{eqJ=1}) and (\ref{eqgradientboundInf}), we infer that (\ref{eqc0}) holds.
 \end{proof}
In turn, we can show the following.

\begin{cor}\label{corsecondrr}
There exists a constant $C>0$ such that
\begin{equation}\label{eqrr}
  |\partial_{rr}\varphi_\infty|\leq C\ \textrm{in}\ \Gamma(2d_0)\setminus \Gamma.
\end{equation}

The normal derivatives of $\mathds{1}_{\Omega_i} \varphi_\infty$, $i=1,2$, exist on $\Gamma$,\begin{equation}\label{eqphinftyrcontlekt}
\textrm{for each}\ s\in \Gamma, \ r\mapsto \partial_r\left(\mathds{1}_{\Omega_i}\varphi_\infty(r,s)\right)
\
\textrm{is continuous up to}\ r=0,\end{equation}
and
\begin{equation}\label{eqc11}
  \partial_r\left(\mathds{1}_{\Omega_2}\varphi_\infty \right)=-\partial_r\left(\mathds{1}_{\Omega_1}\varphi_\infty \right)\ \textrm{on}\ \Gamma.
\end{equation}
\end{cor}

\begin{proof}
  Estimate (\ref{eqrr}) is a consequence of Proposition \ref{proOutC2}, Lemma \ref{lemtangDiffGhost} and Corollary \ref{corgradientBound} (recall also (\ref{eqdiffbmg2})).
The normal derivatives of $\mathds{1}_{\Omega_i}\varphi_\infty$, $i=1,2$, on $\Gamma$ exist and (\ref{eqphinftyrcontlekt}) holds
due to Corollary \ref{corgradientBound}, (\ref{eqrr}) and the Arzela-Ascoli theorem. Lastly, relation (\ref{eqc11}) follows by letting $\delta_i\to 0$, $i=1,2$, in (\ref{eqrefl1}) and using (\ref{eqJ=1}), (\ref{eqrr}).
\end{proof}

We are now in position to prove the main result of this subsection.\begin{pro}\label{prophiinftyshookes}
We have
\[
\varphi_\infty\equiv 0.
\]
\end{pro}

\begin{proof}
Let
\begin{equation}\label{eqpsiyo}
\psi=\mathds{1}_{\Omega_2}\varphi_\infty-\mathds{1}_{\Omega_1}\varphi_\infty.
\end{equation}
The main effort will be placed in establishing that $\psi$ is a weak solution to (\ref{eqlimit}) in $\Omega$. We note that, thanks to Corollary \ref{corgradientBound}, the function  $\psi$ belongs to the Sobolev space $W_0^{1,2}(\Omega)$. Let $\omega\in C_c^\infty(\Omega)$ be arbitrary.

First, we will demonstrate that the   regularity properties of $\varphi_\infty$ that we have obtained thus far are sufficient for the validity of Green's identities:
\begin{equation}\label{eqalexia000-}
\int_{\Omega_i}\nabla \varphi_\infty \nabla \omega dx+\int_{\Omega_i}  \omega \Delta \varphi_\infty dx=\int_{\partial\Omega_i}  \omega  \partial_{\nu_i} \varphi_\infty dS,\ i=1,2,
\end{equation}
where $\nu_i$  stands for the outward unit normal vector to $\partial\Omega_i$.
We will provide the proof only for $i=1$.
Let
\[
\Omega_\delta=\Omega_1\setminus \Gamma(\delta), \ \delta\in (0,d_0).
\]
Clearly,
\begin{equation}\label{eqalexia000}
\int_{\Omega_\delta}\nabla \varphi_\infty \nabla \omega dx+\int_{\Omega_\delta}  \omega \Delta \varphi_\infty dx=\int_{\partial\Omega_\delta}  \omega  \partial_{\nu_\delta} \varphi_\infty dS,
\end{equation}
where $\nu_\delta$  stands for the outward unit normal vector to $\partial\Omega_\delta$.
By the dominated convergence theorem, we find
\begin{equation}\label{eqalexia222}
\int_{\Omega_\delta}\nabla \varphi_\infty \nabla \omega dx=\int_{\Omega_1}
\nabla \varphi_\infty \nabla \omega \mathds{1}_{\Omega_\delta} dx\to \int_{\Omega_1}
\nabla \varphi_\infty \nabla \omega dx \ \textrm{as}\ \delta \to 0,
\end{equation}
and
\begin{equation}\label{eqalexia333}
\int_{\Omega_\delta}  \omega \Delta \varphi_\infty dx=
\int_{\Omega_1}  \omega \Delta \varphi_\infty \mathds{1}_{\Omega_\delta} dx\to
\int_{\Omega_1}  \omega \Delta \varphi_\infty  dx\ \textrm{as}\ \delta \to 0,
\end{equation}
(recall (\ref{eqlimit}) and Corollary \ref{corgradientBound}).
On the other side, as in relation (1.14) of \cite{wei},  we have
\[\int_{\partial\Omega_\delta}  \omega  \partial_{\nu_\delta} \varphi_\infty dS=
\int_{\Gamma}  \omega (\delta,s) \partial_{\nu_1}  \varphi_\infty (\delta,s)\left(1+\mathcal{O}(\delta) \right) dS,\]
where we used that $\nu_\delta(s)=\nu_1(s)$. So, by  Corollaries \ref{corgradientBound}, \ref{corsecondrr} and the dominated convergence theorem, we get
\[\int_{\partial\Omega_\delta}  \omega  \partial_{\nu_\delta} \varphi_\infty dS=-
\int_{\Gamma}  \omega  \partial_r  \varphi_\infty(\delta,s)  dS+\mathcal{O}(\delta)\to
-\int_{\Gamma}  \omega  \partial_r  \varphi_\infty(0,s)  dS=\int_{\Gamma}  \omega  \partial_{\nu_1}  \varphi_\infty  dS.
\]
By letting $\delta\to 0$ in (\ref{eqalexia000}), via (\ref{eqalexia222}), (\ref{eqalexia333}) and the  above two relations, we deduce that (\ref{eqalexia000-}) holds for $i=1$ as desired. The proof for $i=2$ is identical and therefore we omit  it.

Subtracting  (\ref{eqalexia000-}) for $i=1$ from the corresponding equality with $i=2$, recalling Proposition \ref{proOutC2} and (\ref{eqpsiyo}), yields
\begin{equation}\label{eqweak}
\int_{\Omega}\nabla \psi \nabla \omega dx-\int_{\Omega} f_u(w,x)\psi \omega  dx=\int_{\Gamma}  \omega  \left(\partial_{\nu_2} \varphi_\infty-
\partial_{\nu_1} \varphi_\infty\right) dS.
\end{equation}
We observe that, owing to (\ref{eqc11}), we have
\[
\partial_{\nu_2} \varphi_\infty-
\partial_{\nu_1} \varphi_\infty=\partial_{r}\left( \mathds{1}_{\Omega_2}\varphi_\infty(0,s)\right)+
\partial_{r}\left(\mathds{1}_{\Omega_1} \varphi_\infty(0,s)\right)=0\ \textrm{on}\ \Gamma.
\]
Consequently, relation (\ref{eqweak}) implies that $\psi$ is a weak solution to
(\ref{eqlimit}) in $\Omega$. Since the righthand side of (\ref{eqlimit}) is Lipschitz continuous on $\bar{\Omega}$ (recall again (\ref{eqgradientboundInf}) and the assumptions on $f$ and $w$ from the introduction), by standard regularity theory we infer that $\psi\in C^{2,1}(\bar{\Omega})$. Hence, problem (\ref{eqlimit}) is satisfied in the classical sense in $\Omega$ by $\psi$.

Finally, by the nondegeneracy assumption from Definition \ref{defNonDeg} associated to (\ref{ewlinearFul}) (keep in mind (\ref{eqw0}) and that $f$ is odd), we conclude that $\psi\equiv 0$.
Hence, in light of (\ref{eqpsiyo}), we have established the assertion of the proposition.
\end{proof}

\subsection{Uniform vanishing of $\varphi_n$, $\epsilon_n\partial_r\varphi_n$, $\epsilon_n^2\partial_{rr}\varphi_n$ in $\Gamma(d)$, and of $\varphi_n$ in $C^2\left(\Omega\setminus \Gamma(d) \right)$ as $n\to \infty$}\label{subsecSumarizing}
In light of Propositions \ref{proOutC2} and \ref{prophiinftyshookes}, passing to a subsequence if necessary, we infer that the second term in the expression for $\|\varphi_n\|_0$ in (\ref{eqNorm0}) tends to zero, i.e.,
\begin{equation}\label{eqC2mideniki}
\|\varphi_n\|_{C^2\left(\Omega\setminus \Gamma(d)\right)}\to 0\ \textrm{as}\ n\to \infty.
\end{equation}
In what follows we will show that the same holds for the first, fourth and last term in the corresponding righthand side of (\ref{eqNorm0}).

\begin{pro}\label{prouniform}
We have
\[
\varphi_n\to 0\ \textrm{uniformly\ on}\ \overline{\Gamma(d)}\ \textrm{as}\ n\to \infty.
\]
\end{pro}

\begin{proof}
We will argue by contradiction. So, without loss of generality, let us suppose that
along a subsequence there exist $x_n\in \Gamma(d)$ such that
\begin{equation}\label{eqolme1}
\varphi_1(x_n)\geq 1.
\end{equation}

By virtue of Propositions \ref{proexpdec}, \ref{proexchangelemma} and \ref{prophiinftyshookes}, we infer that
there exists an $M\gg 1$ (independent of $n$) such that
\begin{equation}\label{eqolme2}
  |\varphi_1|<1/4\ \textrm{in}\ \Omega\setminus \Gamma\left(  M\epsilon_n\right),
\end{equation}
provided that $n$ is sufficiently large. Hence, in view of (\ref{eqolme1}) and (\ref{eqolme2}), we can write
\begin{equation}\label{eqolme3}
x_n=(r_n,s_n)\ \textrm{with}\ |r_n|\leq C\epsilon_n\ \textrm{and}\ s_n\in \Gamma.
\end{equation}
Passing to a subsequence if needed, we may assume that
\begin{equation}\label{eqolme4}
s_n\to s_*\ \textrm{for some}\ s_*\in \Gamma.
\end{equation}

By (\ref{eqcontraNorm}), (\ref{eqolme1}) and (\ref{eqolme3}), we deduce that
\begin{equation}\label{eqolme5}
\varphi_1\geq 1/2\ \textrm{in}\ \Gamma(C\epsilon)\cap B_\delta(s_*),
\end{equation}
where $B_\delta(s_*)\subset \Gamma(d)$ denotes  the ball with center at $s_*$ and radius $\delta$ (independent of $n$).

By Lemma \ref{lemBU}, passing to a further subsequence if necessary,  we obtain
\begin{equation}\label{eqolme6}
\varphi_1(r,s_*)-c_*V_1'\left(\epsilon^{-1}b(s_*) \left(r- \epsilon \zeta(s_*)\right) \right)\to 0\ \textrm{uniformly on}\ [-M\epsilon,M\epsilon]\ \textrm{as}\ n\to \infty,
\end{equation}
where $M$ as in (\ref{eqolme2}) and $c_*\in \mathbb{R}$ independent of $n$. Due to (\ref{eqolme5}) and the fact that $V_1'>0$, we see that
  $c_*>0$.
 Making again use of (\ref{eqolme6}) but now of the property  $V_1''>0$, we arrive at
\[
\varphi_1(M\epsilon,s_*)-\varphi_1(0,s_*)\geq o(1),\ \textrm{where}\ o(1)\to 0\ \textrm{as}\ n\to \infty.
\]
 However, the above
relation is in contradiction to (\ref{eqolme2}) and (\ref{eqolme5}).
\end{proof}

As a consequence of the above proposition, we have the following.

\begin{cor}\label{corur}
We have
\[
\epsilon \varphi_r\to 0\ \textrm{and}\  \epsilon^2 \varphi_{rr}\to 0\ \textrm{uniformly on}\ \overline{\Gamma(d)}\ \textrm{as}\ n\to \infty.
\]
\end{cor}
\begin{proof}
By combining (\ref{eqinhomog}), (\ref{eqpot1}), (\ref{eqpot2}), (\ref{eqcontraNorm}), (\ref{eqC2mideniki}) and Propositions \ref{proexpdec}, \ref{prouniform}, we get
\begin{equation}\label{eqmagid}
\epsilon^2 \Delta \varphi\to 0\ \textrm{uniformly on}\ \bar{\Omega}\ \textrm{as}\ n\to \infty.
\end{equation}
Standard elliptic estimates (see in particular \cite[Lem. A.1]{bbh24}) now lead to
\begin{equation}\label{eqmagid2}
\epsilon|\nabla \varphi|\to 0\ \textrm{uniformly in}\ \Gamma(2d)\ \textrm{as}\ n\to \infty,
\end{equation}
which clearly implies the validity of the first assertion of the corollary.
The remaining assertion  follows readily from (\ref{eqLaplaceFermi1}), (\ref{eqLaplaceFermi2}), (\ref{eqLaplaceFermi3}), (\ref{eqcontraNorm}), (\ref{eqmagid}) and (\ref{eqmagid2}).
\end{proof}

\begin{rem}\label{remexpdec}
It is easy to see that, armed with Proposition \ref{prouniform}, the proof of Proposition \ref{proexpdec} gives
\[
\varphi_i(x)=o(1)e^{-c\textrm{dist}(x,\Gamma)/\epsilon}+\mathcal{O}(\epsilon^\infty),\ x\in \Omega_j, \ i\neq j,
\]
with $o(1)\to 0$ uniformly in $x$ as $n\to \infty$.
\end{rem}

\subsection{Uniform vanishing of the tangential derivatives in the norm (\ref{eqNorm0}) as $n\to \infty$}\label{subsecTange}

It remains to deal with the tangential derivatives of $\varphi_n$ that appear in (\ref{eqNorm0}). To this end, we will  differentiate (\ref{eqinhomog}) tangentially to $\Gamma$.
We first note that, thanks to the assumptions in Subsection \ref{subsecAnsatz}, we can differentiate (\ref{eqpot1}) and (\ref{eqpot2}) tangentially and get
\begin{equation}\label{eqtangdiffpot}
\left|\partial_{ij}(U_1^2)\right|+\left|\partial_{i}(U_1^2)\right|\leq C(r^2+\epsilon^2),
\end{equation}
\begin{equation}\label{eqtangdiffpot2}\left|\partial_{ij}(U_1U_2)\right|+\left|\partial_{i}(U_1U_2)\right|+\left|\partial_{ij}(U_2^2)\right|+\left|\partial_{i}(U_2^2)\right|\leq C \epsilon^2e^{-cr/\epsilon}+\mathcal{O}(\epsilon^\infty),\end{equation}
for $r\in [0,2d_0)$, $s\in \Gamma$, $i,j=1,\cdots,N-1$, as is easy to check. Analogous estimates hold for $r\in (-2d_0,0]$.

The following lemma will be useful throughout this subsection.
\begin{lem}\label{lemexpdecay}
If $i\neq j$, then
\[
\left|\nabla \varphi_i\right|+\epsilon\left|\nabla^2 \varphi_i\right|+\epsilon^2\left|\nabla^3 \varphi_i\right|=o(1) \epsilon^{-1}e^{-c\textrm{dist}(x,\Gamma)/\epsilon}+\mathcal{O}(\epsilon^\infty), \
x\in\Omega_j\cap \Gamma(3d_0/2).\]
\end{lem}

\begin{proof} We will only consider the case $i=2$, $j=1$ since the other case is identical.
From (\ref{eqinhomog}),  (\ref{eqpot1}), (\ref{eqpot2}), (\ref{eqcontraNorm}) and Proposition \ref{proexpdec}, we obtain
\[
-\Delta \varphi_2=\mathcal{O}\left(\epsilon^{-2}e^{-cr/\epsilon} \right)+\mathcal{O}(\epsilon^\infty),\ r\in (-3\epsilon,2d_0), \ s\in \Gamma.
\]
By Remark \ref{remexpdec}, the above relation and standard interior elliptic estimates (see for instance \cite[Lem. A.1]{bbh24}), we deduce that
\begin{equation}\label{eqs1}
|\nabla \varphi_2|=o\left(\epsilon^{-1}e^{-cr/\epsilon} \right)+\mathcal{O}(\epsilon^\infty),\ r\in (-2\epsilon,2d_0-\epsilon), \ s\in \Gamma.
\end{equation}

We note that, besides of (\ref{eqremexpdecay1}) and (\ref{eqremexpdecay2}), we also have
\begin{equation}\label{eqspud}
\left|\nabla (U_1^2) \right|\leq C(r+C\epsilon), \ r\in (-2\epsilon,2d_0),\ s\in \Gamma.
\end{equation}
For $k=1,\cdots,N$, we let
\begin{equation}\label{eqs2}
v=\partial_{x_k}\varphi_2.
\end{equation}
Then, by differentiating the second equation in (\ref{eqinhomog}) with respect to $x_k$, using (\ref{eqcontraNorm}) and the above, we find
\[
-\Delta v=\mathcal{O}\left(\epsilon^{-3}e^{-cr/\epsilon} \right)+\mathcal{O}(\epsilon^\infty), \ r\in (-2\epsilon,2d_0-\epsilon), \ s\in \Gamma.
\]
Hence, by (\ref{eqs1}), (\ref{eqs2}), the above relation, and standard elliptic estimates,  we infer that
\begin{equation}\label{sloukas1}
|\nabla v|=o\left(\epsilon^{-2}e^{-cr/\epsilon} \right)+\mathcal{O}(\epsilon^\infty), \ r\in (-\epsilon,2d_0-2\epsilon), \ s\in \Gamma.
\end{equation}

By our assumptions in Subsection \ref{subsecAnsatz}, we can differentiate (\ref{eqremexpdecay2}) and (\ref{eqspud}) one more time to get
\begin{equation}\label{eqonemoretine}
\left|\nabla^2(U_1U_2) \right|\leq C e^{-c\textrm{dist}(x,\Gamma)/\epsilon}\ \textrm{and}\ \left|\nabla^2(U_1^2) \right|\leq C, \ r\in (-\epsilon,2d_0),\ s\in \Gamma,
\end{equation}
respectively. For $k,l=1,\cdots,N$, we let
\begin{equation}\label{eqsloukas2}
\tilde{v}=\partial^2_{kl}\varphi_2=\partial_l v.
\end{equation}
Then, thanks to the above,  we can differentiate (\ref{eqinhomog}) one more time to arrive at
\begin{equation}\label{eqcf}
-\Delta \tilde{v}=\mathcal{O}\left(\epsilon^{-4}e^{-cr/\epsilon} \right)+\mathcal{O}(\epsilon^\infty), \ r\in (-\epsilon,2d_0-2\epsilon), \ s\in \Gamma.
\end{equation}
The remaining assertion of the lemma now follows from (\ref{sloukas1}), (\ref{eqsloukas2}), the above relation and standard elliptic estimates.
\end{proof}

As a first implication of the above lemma, we have the following improvement of (\ref{eqC2mideniki}).
\begin{cor}\label{corC3}
For any $\gamma\in (0,1)$, we have
\[
\|\varphi_n\|_{C^{3+\gamma}(\Gamma(3d_0/2)\setminus \Gamma(d))}\to 0\ \textrm{as}\ n\to \infty.
\]
\begin{proof}
Making stronger use of  (\ref{eqcontraNorm}), Lemma \ref{lemexpdecay} and (\ref{eqonemoretine}), we can differentiate twice (\ref{eqinhomog}) with respect to $x$ and find
that
\[
\Delta \left(\partial^2_{x_ix_j}\varphi_n\right) \ \textrm{is uniformly bounded over}\ \Gamma(2d_0)\setminus \Gamma(d/2)\ \textrm{as}\ n\to \infty,
 \]
 for $i,j=1,\cdots,N$ (cf. (\ref{eqcf})). Then, by (\ref{eqC2mideniki}) (applied with $d/2$ in place of $d$), standard  interior elliptic $W^{2,p}$ estimates and Sobolev embeddings, we obtain
 \[
 \partial^2_{x_ix_j}\varphi_n \ \textrm{is uniformly bounded in}\ C^{1+\beta}\left(\Gamma(3d_0/2)\setminus \Gamma(d)\right)\ \textrm{as}\ n\to \infty,
 \]
 for any $\beta \in (0,1)$. The assertion of the corollary now follows at once from  the compact  embedding of the above H\"{o}lder norms and (\ref{eqC2mideniki}).
\end{proof}
\end{cor}

Our first main result in this subsection is the following.

\begin{pro}\label{protang1}
We have
\[
|\nabla_\Gamma\varphi|\to 0\ \textrm{uniformly on}\ \overline{\Gamma(d)}\ \textrm{as}\ n\to \infty.
\]
\end{pro}

\begin{proof}
For a tangential direction $s$ to $\Gamma$, we let
\begin{equation}\label{eqpsi1yo}
\psi=\varphi_s.
\end{equation}
By differentiating (\ref{eqinhomog}) with respect to $s$, making use of (\ref{eqLaplaceFermi1}), (\ref{eqLaplaceFermi2}), (\ref{eqLaplaceFermi3}), (\ref{eqtangdiffpot}) and (\ref{eqtangdiffpot2}), we obtain
\begin{equation}\label{eqtangpdes}\begin{split}
                                     \mathcal{L}(\psi)-g_s =\mathcal{O}(\varphi) + \mathcal{O}(\epsilon^\infty) \ \ \ \   \ \  \ \ \ \ \ \ \ \ \ \ \ \ \ \ \ \ \ \ \ \ \ \ \ \ \ \ \ \ \ \ \ \ \ \ \ \ \ \ \ \ \ \ \ \ \ \ \ \ \ \ \ \ \ \ \ \ \ \ \ \ \ \ \ \ \ \ \ \ \ \ \ \ \ \ \ \  \\
                                      +   \mathcal{O}(\varphi_r)+\mathcal{O}(r)\left(|\nabla_\Gamma \varphi|+|\nabla_\Gamma^2 \varphi| \right)+\epsilon^{-4}\left(\begin{array}{cc}
                                                                                                                                                 \mathcal{O}\left(r^2+\epsilon^2\right) & \mathcal{O}\left(\epsilon^2e^{cr/\epsilon}\right) \\
                                                                                                                                                  \mathcal{O}\left(\epsilon^2e^{cr/\epsilon}\right)& \mathcal{O}\left(\epsilon^2e^{cr/\epsilon}\right)
                                                                                                                                               \end{array}
 \right)\varphi, \ \ \ \ \ \ \ \ \ \ \ \ \
                                  \end{split}
\end{equation}
$r\in(-2d_0,0)$, $s\in \Gamma$. An analogous relation holds for $r\in (0,2d_0)$.

Suppose the assertion of the proposition is false. Then, owing to (\ref{eqC2mideniki}) (which is valid for any small fixed $d$), without loss of generality, there exist $m>0$ (independent of $n$) and $x_n=(r_n,s_n)$ such that
\begin{equation}\label{eqrn}
\psi_1(x_n)=\|\psi_1\|_{L^\infty(\Omega)}\geq m \ \textrm{with}\ r_n\to 0.
\end{equation}

We claim that
\begin{equation}\label{eqbfrac}
\frac{|r_n|}{\epsilon_n}\to \infty.
\end{equation}
Indeed, if not then
\[
\frac{r_n}{\epsilon_n}\to R_* \ \textrm{and}\ s_n\to s_*\ \textrm{for some}\ R_*\in [0,\infty)\ \textrm{and}\ s_*\in\Gamma,
\]
having passed to a subsequence if necessary. As in the proof of Lemma \ref{lemBU}, rescaling
around $(r_n,s_n)$, using (\ref{eqcontraNorm}) and Proposition \ref{prouniform}, passing to a further subsequence if needed, we find that there exists a bounded solution to
the  limit problem
\[\begin{array}{c}
    -\partial_{zz}\Psi_1+V_2^2\left(z+b(s_*)R_*\right)\Psi_1+2V_1V_2\left(z+b(s_*)R_*\right)\Psi_2-b^{-2}(s_*)\Delta_{\mathbb{R}^{N-1}}\Psi_1=0 \\
     \\
    -\partial_{zz}\Psi_2+V_1^2\left(z+b(s_*)R_*\right)\Psi_2+2V_1V_2\left(z+b(s_*)R_*\right)\Psi_1-b^{-2}(s_*)\Delta_{\mathbb{R}^{N-1}}\Psi_2=0
  \end{array}
\]
in $\mathbb{R}^N$ (cf. (\ref{pao13})) such that
\[
\Psi_1(0)\geq m>0 \ \textrm{and}\ \nabla \Psi_1(0)=0.
\]
However, from the last step in the proof of the aforementioned lemma (i.e. applying the main result of \cite{sourdisJDE}), this is absurd since $V_i''>0$, $i=1,2$. Therefore, we have established the validity of (\ref{eqbfrac}).

By virtue of (\ref{eqL}), (\ref{eqpot1}), (\ref{eqpot2}), (\ref{eqcontraNorm}) and the corresponding relation to (\ref{eqtangpdes}) for $r>0$, we have
\begin{equation}\label{eqgiannis1}
-\Delta \psi_1=\mathcal{O}\left(\epsilon^{-2}e^{-cr/\epsilon}\right)+\mathcal{O}(\epsilon^{-1}), \ r\in [0,3d_0/2),\ s\in \Gamma.
\end{equation}
  On the other side,  via (\ref{eqcontraNorm}), Proposition \ref{proexpdec} and Lemma \ref{lemexpdecay}, the first equation  in  (\ref{eqtangpdes}) yields
\begin{equation}\label{eqgiannis2}
-\Delta \psi_1+\left(\epsilon^{-4}U_2^2-f_u(U_1,x)\right)\psi_1=\mathcal{O}\left(\epsilon^{-2}e^{cr/\epsilon}\right)+\mathcal{O}(\epsilon^\infty),
\end{equation}
for $r\in (-3d_0/2,0)$, $s\in \Gamma$. Now, as in Proposition \ref{proexpdec} we get
\begin{equation}\label{eqexpdecaytangs}
\psi_1=\mathcal{O}\left(e^{cr/\epsilon}\right)+\mathcal{O}(\epsilon^\infty), \ r\in (-3d_0/2,0),\ s\in \Gamma.
\end{equation}
In turn, by using the above relation in (\ref{eqgiannis2}) together with (\ref{eqpot1}), we arrive at
\begin{equation}\label{eqgiannis1+}
-\Delta \psi_1=\mathcal{O}\left(\epsilon^{-2}e^{cr/\epsilon}\right)+\mathcal{O}(\epsilon^{\infty}), \ r\in (-3d_0/2,0),\ s\in \Gamma.
\end{equation}

Let
\begin{equation}\label{eqagni1}
\tilde{\Psi}_n(y)=\psi_1(x_n+\epsilon_n x).
\end{equation}
Then, in any fixed large ball $B_K(0)$ we obtain   from (\ref{eqgiannis1}) and (\ref{eqgiannis1+}) that
\[
-\Delta \tilde{\Psi}_n=\mathcal{O}\left(e^{-c\textrm{dist}(x_n,\Gamma)/\epsilon}\right)+\mathcal{O}(\epsilon),\ |\tilde{\Psi}|\leq C, \tilde{\Psi}(y)\leq \tilde{\Psi}(0)\ \textrm{and}\ \tilde{\Psi}(0)\geq m>0.
\]
As in the proof of Lemma \ref{lemBU}, passing to a subsequence if necessary, we have
\begin{equation}\label{eqagni2}
\tilde{\Psi}_n\to \tilde{\Psi}_\infty\ \textrm{in}\ C^1_{loc}(\mathbb{R}^N),
\end{equation}
where $\tilde{\Psi}_\infty \in C^2(\mathbb{R}^N)$ is bounded, $\tilde{\Psi}_\infty(0)\geq m>0$   and
\[
-\Delta \tilde{\Psi}_\infty=0,\   \tilde{\Psi}_\infty(y)\leq \tilde{\Psi}_\infty(0)\ \textrm{in}\ \mathbb{R}^N.
\]
By the strong maximum principle or Liouville's theorem, we deduce that \begin{equation}\label{eqagni3}\tilde{\Psi}_\infty \equiv c>0.\end{equation}

Let $B_\epsilon'(s_n)\subset \mathbb{R}^{N-1}$ and consider $\hat{s}_n,\tilde{s}_n\in \partial B_\epsilon'(s_n)$ be antipodal in the direction $s$.
 An application of the  mean value theorem yields
\[
\left|\varphi_1(r_n,\hat{s}_n)-\varphi_1(r_n,\tilde{s}_n)\right|=\left|\partial_s\varphi_1(r_n,\bar{s}_n)\right|2\epsilon_n\stackrel{(\ref{eqpsi1yo})}{=}\left|\psi_1(r_n,\bar{s}_n)\right|2\epsilon_n,
\]
for some $\bar{s}_n\in B_\epsilon'(0)$ in the line segment connecting $\hat{s}_n$ to $\tilde{s}_n$. On the one hand, recalling Proposition \ref{prouniform}, the lefthand side of the above relation  tends to zero as $n\to \infty$. On the other hand, by (\ref{eqagni1}), (\ref{eqagni2}) and (\ref{eqagni3}),  the righthand side is bounded from below by $c$, provided that $n$ is sufficiently large, where $c$ as in (\ref{eqagni3}).  We have thus reached a contradiction,  completing the proof of the proposition.
 \end{proof}

An important consequence of the above proposition is the following.

\begin{cor}\label{cormix}
We have
\[
\epsilon_n\left|\nabla_\Gamma \varphi_r\right|\to 0,\ \textrm{uniformly on}\ \overline{\Gamma(d)},\ \textrm{as}\ n\to \infty.
\]
\end{cor}

\begin{proof}
This follows readily from Proposition \ref{protang1}, (\ref{eqpsi1yo}), (\ref{eqgiannis1}), (\ref{eqgiannis1+}) and standard elliptic estimates.
\end{proof}

\begin{rem}\label{remexpdecSSS}
We note that thanks to (\ref{eqcontraNorm}), Proposition \ref{prouniform}, Remark \ref{remexpdec}, Lemma \ref{lemexpdecay}, Proposition \ref{protang1} and (\ref{eqtangpdes}), relation (\ref{eqgiannis2}) can be strengthened to
\[
-\Delta \psi_1+\left(\epsilon^{-4}U_2^2-f_u(U_1,x)\right)\psi_1=o\left(\epsilon^{-2}e^{cr/\epsilon}\right)+\mathcal{O}(\epsilon^\infty),
\]
for $r\in (-3d_0/2,0)$ and $s\in \Gamma$. Thus, as in Proposition \ref{proexpdec} (cf. Remark \ref{remexpdec}), making once more use of the assertion of Proposition \ref{protang1}, we deduce that
\[
\psi_1=o(1)e^{cr/\epsilon}+\mathcal{O}(\epsilon^\infty),\ r\in (-3d_0/2,0), \ s\in \Gamma.
\]
An analogous estimate holds for $\psi_2$ if $r\in (0,3d_0/2)$, $s\in \Gamma$.
\end{rem}

The remainder of this subsection will be devoted to the proof of  the following proposition.

\begin{pro}\label{protang2}
For any directions $s,\sigma$ on $\Gamma$, we have
\[
\partial^2_{s\sigma}\varphi\to 0,\ \textrm{uniformly on}\ \overline{\Gamma(d)}, \ \textrm{as}\ n\to \infty.
\]
\end{pro}
\begin{proof}
Let
\begin{equation}\label{eqrhodef}
\rho=\partial^2_{s\sigma}\varphi \stackrel{(\ref{eqpsi1yo})}{=}\partial_\sigma \psi.
\end{equation}
By differentiating (\ref{eqtangpdes}) with respect to $\sigma$, keeping in mind (\ref{eqcontraNorm}), (\ref{eqtangdiffpot}) and (\ref{eqtangdiffpot2}), we get
\begin{equation}\label{eqtangpdessss}\begin{split}
                                     \mathcal{L}(\rho)-g_{s\sigma} = \mathcal{O}(\varphi_\sigma)+\mathcal{O}(\varphi)+\mathcal{O}(\psi)+\mathcal{O}(\epsilon^\infty)    \ \ \ \ \ \ \ \ \ \ \ \ \ \ \ \ \ \ \ \ \ \ \ \ \ \ \ \ \ \ \ \ \ \ \ \ \ \ \ \ \ \ \ \ \ \ \ \ \ \ \ \ \ \ \ \ \ \ \  \\
                                        + \mathcal{O}(\varphi_{r\sigma})+\mathcal{O}(r)\left(|\nabla_\Gamma \varphi_\sigma|+|\nabla_\Gamma^2 \varphi_\sigma| \right)+\epsilon^{-4}\left(\begin{array}{cc}
                                                                                                                                                 \mathcal{O}\left(r^2+\epsilon^2\right) & \mathcal{O}\left(\epsilon^2e^{cr/\epsilon}\right) \\
                                                                                                                                                  \mathcal{O}\left(\epsilon^2e^{cr/\epsilon}\right)& \mathcal{O}\left(\epsilon^2e^{cr/\epsilon}\right)
                                                                                                                                               \end{array}
 \right)\varphi_\sigma \ \ \ \ \ \ \ \ \ \ \ \ \ \\
 +\mathcal{O}(\varphi_r)+\mathcal{O}(r)\left(|\nabla_\Gamma \varphi|+|\nabla_\Gamma^2 \varphi| \right)+\epsilon^{-4}\left(\begin{array}{cc}
                                                                                                                                                 \mathcal{O}\left(r^2+\epsilon^2\right) & \mathcal{O}\left(\epsilon^2e^{cr/\epsilon}\right) \\
                                                                                                                                                  \mathcal{O}\left(\epsilon^2e^{cr/\epsilon}\right)& \mathcal{O}\left(\epsilon^2e^{cr/\epsilon}\right)
                                                                                                                                               \end{array}
 \right)\varphi \ \ \ \ \ \ \ \ \ \ \ \ \ \\
 +\mathcal{O}(\psi_r)+\mathcal{O}(r)\left(|\nabla_\Gamma \psi|+|\nabla_\Gamma^2 \psi| \right)+\epsilon^{-4}\left(\begin{array}{cc}
                                                                                                                                                 \mathcal{O}\left(r^2+\epsilon^2\right) & \mathcal{O}\left(\epsilon^2e^{cr/\epsilon}\right) \\
                                                                                                                                                  \mathcal{O}\left(\epsilon^2e^{cr/\epsilon}\right)& \mathcal{O}\left(\epsilon^2e^{cr/\epsilon}\right)
                                                                                                                                               \end{array}
 \right)\psi, \ \ \ \ \ \ \ \ \ \ \ \ \
                       \end{split}
\end{equation} if $r\in (-2d_0,0]$, $s\in \Gamma$. An analogous relation holds for $r\in (0,2d_0)$.

If $r\in (-3d_0/2,0)$, the righthand side of the first equation   of the above relation, modulo $\mathcal{O}(\epsilon^\infty)$, can be bounded by
\begin{equation}\label{eqrhs}
o(\epsilon^{-2})e^{cr/\epsilon}+Cr\left|\nabla^3_{\Gamma}\varphi \right|\leq o(\epsilon^{-2})e^{cr/\epsilon} +ro(\epsilon^{-3})e^{cr/\epsilon}\leq o(\epsilon^{-2})e^{cr/\epsilon},
\end{equation}
as is easy to check (keep in mind (\ref{eqcontraNorm}), Remark \ref{remexpdec}, Lemma \ref{lemexpdecay} and Remark \ref{remexpdecSSS}). In turn, via   (\ref{eqpot2}) and (\ref{eqcontraNorm}), the first equation of (\ref{eqtangpdessss}) gives
\[
-\Delta \rho_1+\left(\epsilon^{-4}U_2^2-f_u(U_1,x)\right) \rho_1= \mathcal{O}(\epsilon^{-2})e^{cr/\epsilon}+\mathcal{O}(\epsilon^\infty),\ r\in (-3d_0/2,0),\ s\in \Gamma,
\]
(cf. (\ref{eqgiannis2})). Then, similarly to (\ref{eqexpdecaytangs}), we infer that
\begin{equation}\label{eqexpdecaytangsssssssssg}
\rho_1=\mathcal{O}\left(e^{cr/\epsilon}\right)+\mathcal{O}(\epsilon^\infty), \ r\in (-3d_0/2,0),\ s\in \Gamma.
\end{equation}
An analogous estimate holds for $\rho_2$ if  $r\in(0,3d_0/2)$.

Armed with the above information, we can write (\ref{eqtangpdessss}) and the corresponding equation for $r>0$ as
\begin{equation}\label{eqfinal}
-\Delta \rho=\mathcal{O}(r)\mathcal{O}\left(\left|\nabla^3_\Gamma \varphi\right| \right)+\mathcal{O}(\epsilon^{-2})e^{-c|r|/\epsilon}+\mathcal{O}(\epsilon^{-1})\ \textrm{in}\ \Gamma(3d_0/2).
\end{equation}
Hence, by standard interior elliptic estimates (see for instance \cite[Lem. A.1]{bbh24}), using once more that $|\rho|\leq C$, we deduce that
\[\begin{array}{rcl}
    \|\nabla \rho\|_{L^\infty\left(\Gamma(d_0)\right)}^2 & \leq & \delta\|\nabla^3_\Gamma \varphi\|_{L^\infty\left(\Gamma(3d_0/2) \right)}^2 +\delta^{-1}\mathcal{O}(\epsilon^{-4}) \\
     &  &  \\
    & \stackrel{\textrm{Cor.} \ref{corC3}}{\leq} & \delta\|\nabla^3_\Gamma \varphi\|_{L^\infty\left(\Gamma(d_0) \right)}^2 +\delta^{-1}\mathcal{O}(\epsilon^{-4}),
  \end{array}
 \]
 for $\delta\in (0,1)$ independent of $n$. Applying the above estimate for all directions $s,\sigma$ on $\Gamma$ (recall (\ref{eqrhodef})), summing them,  using that
 $\left|\nabla_\Gamma \cdot \right|\leq C\left|\nabla \cdot\right|$, and then choosing a sufficiently small $\delta$, we conclude that
\[
\|\nabla^3_\Gamma \varphi\|_{L^\infty\left(\Gamma(d_0) \right)}\leq C\epsilon^{-2}.
\]

It follows from (\ref{eqtangpdessss}) and the corresponding relation for $r>0$, via (\ref{eqcontraNorm}), (\ref{eqrhs}) together with the justification following it,  and the above estimate, that
\[
\mathcal{L}(\rho)=\mathcal{O}(r)\epsilon^{-2}+o(\epsilon^{-2})\ \textrm{in}\ \Gamma(d_0).
\]
Moreover, (\ref{eqfinal}) becomes
\[
-\Delta \rho=\mathcal{O}(r)\epsilon^{-2}+\mathcal{O}(\epsilon^{-2})e^{-c|r|/\epsilon}+\mathcal{O}(\epsilon^{-1})\ \textrm{in}\ \Gamma(d_0),
\]
(cf. (\ref{eqgiannis1})).

Thanks to the above two relations,  the rest of the proof becomes a straightforward adaptation of that of Proposition \ref{protang1} and  is therefore omitted.
We just remark that the term $\mathcal{O}(r)\epsilon^{-2}$ does not affect the analysis in the proof of the aforementioned proposition since after rescaling (recall that both sides are multiplied   by $\epsilon^2$ in the process) it becomes
$\mathcal{O}(r_n)+\mathcal{O}(\epsilon_n)\to 0$ uniformly over compacts (where $(r_n,s_n)$ the corresponding point to  (\ref{eqrn})).
\end{proof}

\subsection{Completion of the proof of Theorem \ref{thmMain}}\label{subsecfinal}
By virtue of (\ref{eqC2mideniki}), Propositions \ref{prouniform}, \ref{protang1}, \ref{protang2} and Corollaries \ref{corur}, \ref{cormix}, we get
$\|\varphi_n\|_0\to 0$ which contradicts (\ref{eqcontraNorm}).

\appendix  \section{Useful a-priori estimates for a Schr\"{o}dinger operator}\label{ApendixA}
In this appendix, we will prove the following result which is needed in the proof of Proposition \ref{proexpdec} as explained in Remark \ref{remGenScrod}.
\begin{lem}\label{lemA1}
There exist positive constants $\epsilon_0$, $C$ such that if $\phi\in C^2(\mathcal{S})\cap C(\bar{\mathcal{S}})$, $g\in C(\mathcal{S})$ satisfy
\begin{equation}\label{eqappenSchrod1}
-\epsilon^4 \Delta \phi+w_1^2\phi=g\ \textrm{in}\ \mathcal{S};\ \phi=0\ \textrm{on}\ \partial\mathcal{S},
\end{equation}
with $\epsilon \in (0,\epsilon_0)$ and $w_1$, $\mathcal{S}$ as in the proof of Proposition \ref{proexpdec}, then
\[
\|\phi\|_{L^\infty(\mathcal{S})}\leq C\epsilon^{-2}\|g\|_{L^\infty(\mathcal{S})}.
\]

Moreover, under the same regularity conditions on $\phi$, $g$,  assuming that
\begin{equation}\label{eqappenSchrod2}
-\epsilon^4 \Delta \phi+w_1^2\phi=\textrm{dist}^2(x,\partial \mathcal{S}\cap \partial \Omega)g\ \textrm{in}\ \mathcal{S};\ \phi=0\ \textrm{on}\ \partial\mathcal{S},
\end{equation}
 then
\[
\|\phi\|_{L^\infty(\mathcal{S})}\leq C\|g\|_{L^\infty(\mathcal{S})}.
\]
\end{lem}
\begin{proof}We start by proving the first assertion of the lemma. To this end, we will argue by contradiction. So, let us suppose that there exist sequences
$\epsilon_n \to 0$, $\phi_n\in C^2(\mathcal{S})\cap C(\bar{\mathcal{S}})$, $g_n\in C(\mathcal{S})$ satisfying (\ref{eqappenSchrod1}) such that
\begin{equation}\label{eqappencontranorms}
\|\phi_n\|_{L^\infty(\mathcal{S})}=1\ \textrm{and}\ \epsilon_n^{-2}\|g_n\|_{L^\infty(\mathcal{S})}\to 0.
\end{equation}

Without loss of generality, we may assume that
\[
\phi_n(x_n)=\|\phi_n\|_{L^\infty(\mathcal{S})}=1\ \textrm{for some}\ x_n\in \mathcal{S},\ \textrm{and therefore}\ \Delta \phi_n(x_n)\leq 0.
\]
Now, substituting $x=x_n$ in (\ref{eqappenSchrod1}) gives
\[
w_1^2(x_n)\leq g_n(x_n),
\
\textrm{and via (\ref{eqappencontranorms}) that}\
\epsilon_n^{-2} w_1^2(x_n)\to 0.
\]
Thus, recalling (\ref{eqhopf24}), we have
\begin{equation}\label{eqappenloc}
\epsilon_n^{-1} \textrm{dist}(x_n, \partial \mathcal{S}\cap \partial \Omega)\to 0.
\end{equation}

 It follows from (\ref{eqappenSchrod1}), (\ref{eqappencontranorms}) and the fact that $w_1$ vanishes on $\partial \Omega\cap \partial\mathcal{S}$ that
 \[
 \Delta \phi_n=\mathcal{O}(\epsilon_n^{-2})\ \textrm{if}\ \textrm{dist}(x,\partial \mathcal{S}\cap \partial \Omega)<C\epsilon_n.
 \]
 Stretching variables, $x\mapsto (x-p)/\epsilon_n$ with $p\in \partial \mathcal{S}\cap \partial \Omega$, in the above relation; applying standard boundary elliptic estimates in the resulting one (using once more (\ref{eqappencontranorms})), and then scaling these back yields
 \[
 |\nabla \phi_n|\leq C\epsilon_n^{-1}\ \textrm{in}\ B_p(C\epsilon_n)\cap \mathcal{S},
 \]
 where $B_p(r)$ denotes the ball of radius $r$ centered at $p$.
 In fact, the generic constant $C$ in the above estimate can be taken to be independent of $p$ as well (due to the smoothness and compactness of $\partial \mathcal{S}\cap \partial \Omega$). Thus, we obtain
 \[
 |\nabla \phi_n|\leq C\epsilon_n^{-1}\ \textrm{if}\ \textrm{dist}(x,\partial \mathcal{S}\cap \partial \Omega)\leq C\epsilon_n.
 \]

However, the above relation is in contradiction with   (\ref{eqappencontranorms}), (\ref{eqappenloc}) and the fact that $\phi_n$ vanishes on  $\partial \mathcal{S}\cap \partial \Omega$ (just apply the mean value theorem between $x_n$ and its closest point on $\partial \mathcal{S}\cap \partial \Omega$).
 The proof of the first assertion of the lemma is  complete.

The proof of the second assertion proceeds along the same lines but is   more direct. It turns out that now there is no need for rescaling and employing elliptic boundary estimates.
Indeed, using the previous notation, we can conclude at once by setting $x=x_n$ in (\ref{eqappenSchrod2}) and using (\ref{eqhopf24}).
\end{proof}

\end{document}